\documentclass[12pt]{amsart}

\usepackage{amsmath,amssymb}
\usepackage{amsthm}
\usepackage{hyperref}
\usepackage[margin=1.0in]{geometry}

\numberwithin{equation}{section}

\newtheorem{prop}{Proposition}

\newtheorem{thm}[prop]{Theorem}
\newtheorem{cor}[prop]{Corollary}

\numberwithin{prop}{section}

\theoremstyle{definition}
\newtheorem{defn}[prop]{Definition}

\newtheorem{rmk}[prop]{Remark}

\newcommand{\lp}[1]{\left( #1 \right)}
\newcommand{\lb}[1]{\left[ #1 \right]}
\newcommand{\lo}[1]{\left| #1 \right|}

\newcommand{\del}{\partial}
\newcommand{\dt}{\frac{\partial}{\partial t}}
\newcommand{\brs}[1]{\left| #1 \right|}

\newcommand{\gD}{\Delta}
\newcommand{\gd}{\delta}

\newcommand{\ga}{\alpha}
\newcommand{\gb}{\beta}

\renewcommand{\bar}[1]{\overline{#1}}

\newcommand{\IP}[1]{\left<#1\right>}

\DeclareMathOperator{\Rc}{Rc}
\DeclareMathOperator{\Rm}{Rm}

\DeclareMathOperator{\tr}{tr}

\DeclareMathOperator{\diam}{diam}

\begin{document}

\title[Four-dimensional generalized Ricci flows with nilpotent symmetry]{Four-dimensional generalized Ricci flows with nilpotent symmetry}

\begin{abstract} We study solutions to generalized Ricci flow on four-manifolds with a nilpotent, codimension $1$ symmetry.  We show that all such flows are immortal, and satisfy type III curvature and diameter estimates.  Using a new kind of monotone energy adapted to this setting, we show that blowdown limits lie in a canonical finite-dimensional family of solutions.  The results are new for Ricci flow.
\end{abstract}

\author{Steven Gindi}
\address{Whitney Hall\\
                Binghamton University\\
                Binghamton, NY 13902}
\email{\href{mailto:sgindi@binghamton.edu}{sgindi@binghamton.edu}}

\author{Jeffrey Streets}
\address{Rowland Hall\\
         University of California, Irvine\\
         Irvine, CA 92617}
\email{\href{mailto:jstreets@uci.edu}{jstreets@uci.edu}}

\date{June 17, 2021}

\thanks{
JS was supported by the NSF via DMS-1454854.
}

\maketitle

\section{Introduction}

Given a smooth manifold $M$, a one-parameter family of pairs $(g_t, H_t)$ of Riemannian metrics and closed three-forms is a solution of \emph{generalized Ricci flow} (cf. \cite{garciafernStreets2020}) if
\begin{align*}
    \dt g =&\ -2 \Rc + \frac{1}{2} H^2,\\
    \dt H =&\ \gD_{d,g_t} H,
\end{align*}
where $H^2(X,Y) = \IP{i_X H, i_Y H}$, and $\gD_{d,g_t}$ denotes the Hodge Laplacian associated to $g_t$.  This system of equations is a natural extension of the Ricci flow equation with applications to complex geometry \cite{SG, PCF, PCFReg}, generalized geometry \cite{GF19,SeveraValach1,StreetsTdual,GKRF} and mathematical physics \cite{Friedanetal,OSW}.  A standard theme in understanding evolution equations is to analyze the long-time limiting behavior.  For Ricci flow, many global solutions exhibit collapse with bounded curvature after taking appropriate blowdown limits, and  such families of metrics acquire nilpotent symmetry groups (cf. \cite{CFG, Lott1, LottDR}, also recent work \cite{huang2020ricci}).  Our purpose in this paper is to describe solutions to generalized Ricci flow in four dimensions admitting a three-dimensional nilpotent symmetry group.  This class of solutions represents the lowest-dimensional setting, outside of local homogeneity, where a nonabelian, nilpotent symmetry can occur, and arises naturally in understanding blowdown limits of type III solutions in four dimensions.  Our first main result shows the global existence and type III curvature decay and diameter growth for solutions to generalized Ricci flow with this symmetry (cf. Theorem \ref{t:lte3+1} below).

\begin{thm} \label{t:mainthm1}
Let $(P, \bar{g}, \bar{H})$ be a Riemannian metric and closed three-form on the nilpotent, four-dimensional $\mathcal G$-principal bundle $P \to M$, where $M$ is compact and one-dimensional.  Then the solution to generalized Ricci flow with initial data $(\bar{g}, \bar{H})$ exists on $[0,\infty)$, and there is a universal constant $C$ so that the solution satisfies the type III estimates
\begin{align*}
    \sup_{M \times [0,\infty)} t\left( \brs{\bar{\Rm}} + \brs{\bar{H}}^2 \right) \leq&\ C, \qquad 
    \diam(g(t)) \leq \diam(g(0)) t^{\frac{1}{2}}.
\end{align*}
\end{thm}

\begin{rmk} 
\begin{enumerate}
\item This result is new even in the case of Ricci flow ($H \equiv 0$).
\item In the statement above $g_t$ is the one-parameter family of metrics on $M$ induced by the family of invariant metrics on $P$ (cf. \S \ref{s:SAB})
    \item The constant $C$ above is computable from our proof, although we do not carry this out.
\end{enumerate}
\end{rmk}

The proof of Theorem \ref{t:mainthm1} relies principally on several surprising maximum principles for certain first-order geometric quantities associated to the flow $(\bar{g}_t, \bar{H}_t)$.  By general theory, the solution to generalized Ricci flow will preserve the symmetry, and yields a time-dependent family of invariant metrics described by a metric $g$ on $M$, a family of metrics $G$ on the fibers, and a principal connection $A$.  These decomposed flow equations were derived in \cite{GindiStreets}, and we recall them and their canonical gauging  in \S \ref{s:SAB}.  In the setting of Theorem \ref{t:mainthm1}, we first show in Proposition \ref{PROPbra2} that the $G$-norm of the Lie bracket is a subsolution of the associated heat equation.  The square norm of the projection of $H$ onto the vertical space also obeys such an evolution equation, as does $\brs{DG}^2 + \tr_g H^2$.  These favorable evolution equations are in fact strong enough to imply $O(t^{-1})$ decay for these quantities with sharp constants (cf. Corollary \ref{c:typeIIIestimates}).  Using these sharp decay estimates we derive the general type III estimate claimed in Theorem \ref{t:mainthm1}, which in turn implies the global existence.

To address the limiting behavior at infinity, we first note that due to the type III curvature bound, it is natural to expect that (possibly collapsing) blowdown limits at infinity will exist by adapting compactness results for Ricci flow (\cite{HamiltonCompactness, Lott1,LottDR}) to generalized Ricci flow.  In light of the $O(t^{\frac{1}{2}})$ diameter growth, these blowdown limits will either collapse the base manifold to a point or converge to a solution defined again over a circle.  The second main result classifies the limits in the latter case.  We give here an informal statement, referring to Theorem \ref{THMLIMIT} and Corollary \ref{CORLIMIT} for the precise claims.

\begin{thm} \label{t:convthm} Let $(P, \bar{g}(t), \bar{H}(t))$ denote a solution to generalized Ricci flow as in Theorem \ref{t:mainthm1}.  Suppose $(P_{\infty}, \bar{g}_{\infty}(t), \bar{H}_{\infty}(t))$ is a blowdown limit at infinity, where $P_{\infty}$ is a four-dimensional $\mathcal G$-principal bundle $P_{\infty} \to M_{\infty}$, where $M$ is compact and one-dimensional.  Then $(P_{\infty}, \bar{g}_{\infty}(t), \bar{H}_{\infty})$ belongs to a canonical finite-dimensional family of solutions.
\end{thm}

\begin{rmk} The canonical family above is essentially explicit, and we describe it here in the case of nonabelian ($\mbox{Nil}_3$) structure group.  The base metric $g$ expands homothetically in time, while the principal connections associated to $\bar{g}$ are fixed in time.  The three-form $\bar{H}$ is identified with a three-form on the Lie algebra which is fixed in time, and has norm constant in space.   At all times the fiber metric $G$ defines a harmonic map on the base space, and evolves according to a system of ODEs governing locally homogeneous solutions to generalized Ricci flow on $\mbox{Nil}_3$.  In the case of Ricci flow $(H \equiv 0)$ the geometric properties of these solutions are described in \cite{isenbergJackson}, whereas in the general case they appeared in the recent work \cite{Paradiso}.
\end{rmk}

\begin{rmk} Immortal solutions as in Theorem \ref{t:convthm} arise naturally in the general (i.e. without symmetry ansatz) analysis of blowdown limits of type III solutions to generalized Ricci flow on four-manifolds (cf. \cite{Lott1} Theorem 5.12), and blowdown limits in this symmetry class should in fact lie in the canonical family described in Theorem \ref{t:convthm}.
\end{rmk}

The proof of Theorem \ref{t:convthm} relies on a new monotone energy.  First recall that noncollapsed blowdown limits for generalized Ricci flow can be addressed using the expander entropy functional \cite{Streetsexpent, FIN}.  Modifications of the Ricci flow expanding entropy, to account for collapsing, were discovered in \cite{LottDR}, and used there to classify type III solutions to Ricci flow on three-manifolds with $O(t^{\frac{1}{2}})$ diameter growth.  These functionals were adapted to the setting of generalized Ricci flow in \cite{GindiStreets}, and used to prove rigidity and classification results for certain blowdown limits.  These functionals are defined on the base space of a principal bundle with \emph{nilpotent} structure group, although ultimately the key monotonicity formulae only hold for the case of an \emph{abelian} structure group, and the applicability of these functionals in the general case remains unclear.  In the setup of Theorem \ref{t:mainthm1}, the first case of nonabelian nilpotent symmetry (aside from locally homogeneous solutions on three-manifolds), building on the sharp maximum principles discussed above, we define a new monotone functional (cf. Definition \ref{d:energy}), which leads to rigidity of blowdown limits and the proof of Theorem \ref{t:convthm}.

\vskip 0.1in

\textbf{Acknowledgements}: The authors would like to thank John Lott for useful conversations.

\section{Setup and Background} \label{s:SAB}
\subsection{Differential and Geometric Structures on \texorpdfstring{$\mathfrak{G} \oplus TM$}{}} \label{SECDGS}

Let $(\bar{g},\bar{H})$ be an invariant metric and three-form on a $\mathcal{G}$--principal bundle $\pi:P \rightarrow M$ and let $\mathcal{E}=  \mathfrak{G} \oplus TM$, where $\mathfrak{G}$ is the adjoint bundle. We will define certain differential and geometric structures on $\mathcal{E}$
by using a correspondence between sections of $\mathcal{E}$ and invariant vector fields on $P$. To describe it, consider the isomorphism $\delta: TP \rightarrow \pi^{*}\mathcal{E}$ given by $\delta(Z)=  \{p,\bar{A}Z \} + \pi_{*}Z ,$ where $Z \in T_{p}P$ and $\bar{A}$ is the connection on $P$ associated with $\bar{g}$. If $e \in \Gamma(\mathcal{E})$ then 
 $\tilde{e}:=\delta^{-1}\pi^{*}e$ is an invariant section of $TP$, where $\pi^{*}e \in \Gamma(\pi^{*}\mathcal{E})$ is the pullback section of $e$. Moreover, if $Z$ is an invariant section of $TP$ then it is straightforward to show that $Z= \tilde{e}$ for some unique $e \in \Gamma(\mathcal{E})$.  This correspondence extends naturally to invariant sections of $(\otimes^{k}T^{*}P)\otimes (\otimes^{l}TP) $
and sections of $(\otimes^{k}\mathcal{E^{*}})\otimes (\otimes^{l}\mathcal{E})$. 

 Given the invariant structures $\bar{g}$ and $\bar{H}$ on $P$, this correspondence respectively yields the fiberwise metric $g_{_{\mathcal{E}}}=G\oplus g$ on $\mathcal{E}=\mathfrak{G} \oplus TM$ and  $H \in \Gamma(\Lambda^{3}\mathcal{E}^{*})$. We also obtain differential and  Lie algebroid structures on $\mathcal{E}$. First consider 

\begin{defn}
For $e_{1}, e_{2} \in \Gamma(\mathcal{E})$, define $[e_{1},e_{2}]$ to be the unique section of $\mathcal{E}$ that satisfies \[ \widetilde{[e_{1},e_{2}]}= [\tilde{e}_{1},\tilde{e}_{2}]_{Lie},\] where $[,]_{Lie}$ is the Lie bracket on $\Gamma(TP)$. 
\end{defn}

If we define $ \tau: \mathcal{E}\longrightarrow TM$ by $\tau(\eta + v)= v $, where $\eta \in \mathfrak{G}$,  we then have \cite{GindiStreets}: 

\begin{prop} \label{PROPLIE} Let $e_{1},e_{2} \in \Gamma(\mathcal{E})$ and $f\in C^{\infty}(M)$. 
\begin{itemize}
\item [1)] $\tau([e_{1},e_{2}]) = [\tau(e_{1}),\tau(e_{2})]_{Lie}$.
\item [2)] $[e_{1},fe_{2}]= f[e_{1},e_{2}] +\tau(e_{1})[f]e_{2}$.
\item [3)] $[,]$ on $\Gamma(\mathcal{E})$ satisfies the Jacobi identity. 
\end{itemize}
In other words, $\tau:(\mathcal{E}, [,]) \longrightarrow (TM, [,]_{Lie})$ is a Lie algebroid structure on $\mathcal{E}$.
\end{prop}
\begin{cor}
$D_{v}\eta:= [v,\eta]$ defines a connection $D$ on $\mathfrak{G}$.
\end{cor}
 Let $F$ be the section of $\Lambda^{2}T^{*}M\otimes \mathfrak{G}$ that corresponds to the curvature $\bar{F}:=d\bar{A}+\frac{1}{2}[\bar{A},\bar{A}]$ using $\delta$. The following gives more properties of the Lie algebroid bracket  \cite{GindiStreets}.

\begin{prop} \label{PROPBRA} Let $s:U \longrightarrow P$ be a local section of $P$, $x,y \in \mathfrak{g}$ and $v,w \in \Gamma(TM)$.
\begin{itemize}
\item[1)] $[,]$ restricts to a Lie bracket on each fiber of $\mathfrak{G}$.
\item[2)] $[\{s,x\}, \{s,y\}]=-\{s,[x,y]\}.$
\item[3)] $D_{v}\{s,x\} =\{s, [(s^{*}\bar{A})v, x]\}.$
\item[4)] $[v,w]= [v,w]_{Lie} - F(v,w)$.
\end{itemize}
\end{prop}
Let $[\mathfrak{G},\mathfrak{G}]= \cup_{m \in M} [ \mathfrak{G}|_{m},\mathfrak{G}|_{m}]$, $Z(\mathfrak{G})=\cup_{m \in M} Z(\mathfrak{G}|_{m})$,  where $Z(\mathfrak{G}|_{m})$ is the center of the Lie algebra $(\mathfrak{G}|_{m}, [,])$, and let $Z(\mathfrak{g})$ be the center of the Lie algebra of $\mathcal{G}$.

\begin{prop} \label{PROPZ}
\mbox{}
\begin{itemize}
\item[1)] $D$ restricts to a connection on $Z(\mathfrak{G})$, which is a vector bundle over $M$.
\item[2)] If $x \in Z(\mathfrak{g})$ and $s$ is a local section of $P$ then 
     $\{s,x\}$ is a local section of $Z(\mathfrak{G})$ and $D\{s,x\}=0$.
\item[3)] When $\mathcal{G}$ is nonabelian, nilpotent and three-dimensional, $Z(\mathfrak{G})=[\mathfrak{G},\mathfrak{G}]$.     
\end{itemize}
\end{prop} 
\begin{proof}

First note that, by Part 2), which is proved below, $Z(\mathfrak{G})$ is a vector bundle over $M$. Now let $\eta$ and $\eta'$ be sections of $Z(\mathfrak{G})$ and $\mathfrak{G}$, respectively. Using Part 3) of Proposition \ref{PROPLIE}, we have \[ [D\eta,\eta']=D[\eta,\eta'] - [\eta,D\eta']=0,\]
which shows that $D$ restricts to a connection on $Z(\mathfrak{G})$.

To prove Part 2), let $s$ be a local section of $P$, $x \in Z(\mathfrak{g})$ and $y \in \mathfrak{g}$. Using Proposition \ref{PROPBRA}, $[\{s,x \}, \{s,y\}]= -\{s, [x,y]\}=0$, which shows that $\{s,x\}$ is a local section of $Z(\mathfrak{G})$. 
Moreover, given $v \in TM$, $D_{v}\{s,x\}=\{s, [(s^{*}\bar{A})v, x] \}= 0,$ where we used Proposition \ref{PROPBRA}.

Lastly, in the case when $\mathcal{G}$ is nonabelian, nilpotent and three-dimensional, $ Z(\mathfrak{G})= [\mathfrak{G},\mathfrak{G}]$ because $(\mathfrak{G}|_{m},[,])$ is isomorphic to the three-dimensional Heisenberg algebra, for each $m \in M$. 
\end{proof}

Using the fiberwise metric $G$ on $\mathfrak{G}$, we have 

\begin{prop} \label{PROPSBASIS}
In the case when $\mathcal{G}$ is nonabelian, nilpotent and three-dimensional, there exists an orthonormal frame  $\{\eta_{i}\}_{i=1,2,3}$ for $(\mathfrak{G},G)$ such that $\eta_{3}$ and $[\eta_{1}, \eta_{2}]$ are nowhere zero local sections of $Z(\mathfrak{G})$.
\end{prop}

\begin{proof}
Choose an orthonormal frame $\{\eta_{i}\}_{i=1,2,3}$ for $(\mathfrak{G},G)$ such that $\eta_{3}$ is a local section of  $Z(\mathfrak{G})$. Using Part 3) of Proposition \ref{PROPZ} and that $\mathcal{G}$ is nonabelian, $[\eta_{1}, \eta_{2}]$ is a nowhere zero local section of $[\mathfrak{G},\mathfrak{G}]= Z(\mathfrak{G})$. 
\end{proof}

Given the Lie algebroid structure $\tau:(\mathcal{E}, [,]) \longrightarrow (TM, [,]_{Lie})$, we also define the following exterior derivative which squares to zero.

\begin{defn} \label{DEFD1}
Define $d: \Gamma(\Lambda^{k}\mathcal{E^{*}})\longrightarrow \Gamma(\Lambda^{k+1}\mathcal{E^{*}})$ by 
\begin{align*}
d\sigma(e_{1},...,e_{k+1})=&  \sum_{1\leq i\leq k+1}(-1)^{i-1}\tau(e_{i})[\sigma(e_{1},...,\hat{e_{i}},...,e_{k+1})]   \\   
 &\  +  \sum_{1\leq i<j \leq k+1} (-1)^{i+j}\sigma([e_{i},e_{j}],e_{1},...,\hat{e_{i}},...,\hat{e_{j}},...,e_{k+1}).
\end{align*}
\end{defn}

In the upcoming sections we will need the following three results.
\begin{prop} \label{PROPDB1} Let $B \in \Gamma(\Lambda^{2}\mathcal{E}^{*})$ and $e_{i}= \eta_{i}+v_{i} \in \Gamma(\mathcal{E}= \mathfrak{G} \oplus TM)$.
\begin{align*}
 dB(e_{1},e_{2},e_{3})= D_{v_{1}}B(e_{2},e_{3}) +B(F(v_{1},v_{2}),e_{3}) -B([\eta_{1},\eta_{2}],e_{3}) + \mathrm{cyclic}(1,2,3).
\end{align*} 
\end{prop}
The proof of this proposition is given in \cite{GindiStreets}. Next, we have
\begin{prop} \label{PROPdH}Suppose $\mathcal{G}$ is nilpotent and three-dimensional and $H \in \Gamma(\Lambda^{3}\mathcal{E}^{*})$. Then
 \[dH(v,\eta_{1},\eta_{2},\eta_{3})=D_{v}H(\eta_{1},\eta_{2},\eta_{3}),\] where $v \in TM$ and $\eta_{i} \in \mathfrak{G}$.
 \end{prop}
 \begin{proof}Let  $\{\eta_{i}\}_{i=1,2,3}$  be a local frame for $\mathfrak{G}$ such that $\eta_{3}$ is a local section of  $Z(\mathfrak{G})$. Consider  
 \begin{align*}
 dH(v,\eta_{1},\eta_{2},\eta_{3})
 &=v[H(\eta_{1},\eta_{2},\eta_{3})] 
 -H([v,\eta_{1}],\eta_{2},\eta_{3}) + H([v,\eta_{2}],\eta_{1},\eta_{3})
  -H([v,\eta_{3}],\eta_{1},\eta_{2}) 
   \\& \ \ \  -H([\eta_{1},\eta_{2}],v,\eta_{3}) + H([\eta_{1},\eta_{3}],v,\eta_{2})  -H([\eta_{2},\eta_{3}],v,\eta_{1})
  \\&= v[H(\eta_{1},\eta_{2},\eta_{3})] 
 -H([v,\eta_{1}],\eta_{2},\eta_{3}) + H([v,\eta_{2}],\eta_{1},\eta_{3})
  -H([v,\eta_{3}],\eta_{1},\eta_{2}),
   \end{align*}
 where we used $[\eta_{1},\eta_{2}]$ and $\eta_{3}$ lie in $Z(\mathfrak{G})$, which is of rank one when $\mathcal{G}$ is nonabelian.
 
 This then equals 
 \begin{align*} 
&v[H(\eta_{1},\eta_{2},\eta_{3})] 
 -H(D_{v}\eta_{1},\eta_{2},\eta_{3}) + H(D_{v}\eta_{2},\eta_{1},\eta_{3})
  -H(D_{v}\eta_{3},\eta_{1},\eta_{2})
  \\&=D_{v}H(\eta_{1},\eta_{2},\eta_{3}),
   \end{align*}
   where we used $D_{v}\eta=[v,\eta],$ by the definition of $D$.
 \end{proof}

The next result shows that the trace of $DG$, denoted by $DG(\cdot,\cdot):= \tr_{G}DG$, is $d(\ln h)$ for some function $h$ on $M$ \cite{GindiStreets}. 

\begin{prop} \label{propGRAD} Let $\bar{g}$ be an invariant metric on a nilpotent $\mathcal{G}$--principal bundle $\pi: P \rightarrow M$ and let $\{ x_{i}\}$ be a basis for $\mathfrak{g}$. 
\begin{enumerate}
\item [1)] Given $m \in M$, $\det G(\{p,x_{i} \}, \{p,x_{j} \})= \det G(\{p',x_{i} \}, \{p',x_{j} \})$, for all $p, p' \in \pi^{-1}(m)$.
\item[2)] $h|_{m}=\det G(\{p,x_{i} \}, \{p,x_{j} \})$, for $p \in \pi^{-1}(m)$, defines a global function $h$ on M. 
\item[3)] $DG(\cdot,\cdot)=d(\ln h)$.
\end{enumerate}
\end{prop}

\subsection{Invariant generalized Ricci flow}

Let $(\bar{g}(t), \bar{H}(t))$ be a time-dependent, invariant metric and three-form on a $\mathcal{G}$--principal bundle $\pi: P\rightarrow M$.  As in Section \ref{SECDGS}, consider the  isomorphisms $\delta(t): TP \rightarrow \pi^{*}\mathcal{E}$ given by $\delta(Z)= \{p,\bar{A}Z \} + \pi_{*}Z,$ where $Z \in T_{p}P$ and $\bar{A}(t) $ are the  connections associated with $\bar{g}(t)$.  Using $\delta(t)$, we obtain the following time dependent data: a fiberwise metric  $g_{_{\mathcal{E}}}(t)=G(t) \oplus g(t)$ on $\mathcal{E}= \mathfrak{G}\oplus TM$, a section $H(t)$ of $\Lambda^{3}\mathcal{E}^{*}$, a bracket $[,]$ on $\Gamma(\mathcal{E})$ and the associated operator $d$ and connection $D$, which were defined above for a fixed time. In some places,  we will extend $D$ to a connection on $\mathcal{E}= \mathfrak{G} \oplus TM$ by setting $D=D \oplus  \nabla^{L}$, where $\nabla^{L}$ is the Levi Civita connection of $g$.  To state some other notation, $\frac{\partial A}{\partial t}$ is the time dependent section of $T^{*}M \otimes \mathfrak{G}$ that is associated with $\frac{\del \bar{A}}{\del t}$. Similarly, $F(t) \in \Gamma(\Lambda^{2}T^{*}M\otimes \mathfrak{G})$ corresponds to the curvature $\bar{F}(t):=d\bar{A}+\frac{1}{2}[\bar{A},\bar{A}]$.

Now let $(\bar{g}(t),\bar{H}(t))$ be a solution to generalized Ricci flow on a nilpotent $\mathcal{G}$--principal bundle $P \rightarrow M$.  We will modify this family using diffeomorphisms of $P$.  Let $q(t)=- \tfrac{1}{2} g^{-1} D G(\cdot, \cdot) \in \Gamma(TM)$ and $X(t)=\gd^{-1} \pi^* q \in \Gamma(TP)$, where $\gd^{-1} \pi^*q$ is the horizontal lift of $q$ with respect to the connection $\bar{A}$ associated with $\bar{g}$. Then define $\hat{g}(t)=\psi^{*}\bar{g}(t)$ and $\hat{H}(t)=\psi^{*}\bar{H}(t)$, where $\psi_{t}$ is the  flow of $X(t)$ on $P$ such that $\psi_{t=0}(p)=p$, for all $p \in P$. By a standard computation, $\hat{g}(t)$ and $\hat{H}(t)$ satisfy the following equations, which we will refer to as the generalized Ricci flow equations \emph{in the canonical gauge}: 
 \begin{equation*}
\frac{\partial \hat{g}}{\partial t} =\ -2\Rc_{\hat{g}} + \frac{1}{2} \hat{H}^2 + \mathcal{L}_{(\psi^{-1})_{*}X}\hat{g} \hspace{15mm}
\frac{\partial \hat{H}}{\partial t}  =\ \Delta_{\hat{g}} \hat{H} + d i_{(\psi^{-1})_{*}X}\hat{H}.
\end{equation*}

Given such a solution, we obtain decomposed data $(G(t),g(t),H(t), \frac{\partial A}{\partial t})$, as defined above.  Before stating the relevant evolution equations we first recall a notational convention. Throughout the paper, we will adopt an abstract notation for taking traces of tensor quantities.  For instance, we set
\begin{align*}
D_{\cdot_{3}}G(\cdot_{1},\cdot_{2})D_{\cdot_{3}}G(\cdot_{1},\cdot_{2}) := \tr_{G}^{(1,2)}\tr_{g}^{3}D_{\cdot_{3}}G(\cdot_{1},\cdot_{2})D_{\cdot_{3}}G(\cdot_{1},\cdot_{2}) = G^{ij} G^{kl} g^{\ga \gb} D_{\ga} G_{ik} D_{\gb} G_{jl}.
\end{align*}
Note that we will always trace over $\cdot$'s, while we will use *'s to denote the different components of tensors.  For clarity, we will at times include the traces in the equations.

For the equations below it turns out to be simpler to express the appropriate pullbacks of the time derivatives of $\bar{H}$ rather than compute the time derivative of $H(t)$ directly.  To that end, in the equations below, defined for tensors over $M$, note that $p \in \pi^{-1}(m)$ is arbitrary in the application of $\gd_p^{-1}$.  We have (cf. \cite{GindiStreets}):

\begin{gather} \label{EQ0}
\begin{split} 
 \frac{\partial G}{ \partial t}(\eta_{1},\eta_{2})= &\ (D_{\cdot}DG)_{\cdot}(\eta_{1},\eta_{2}) -D_{\cdot_{2}}G(\eta_{1},\cdot_{1})D_{\cdot_{2}}G(\cdot_{1},\eta_{2}) \\ &\  -\frac{1}{2}G(F(\cdot_{1},\cdot_{2}), \eta_{1})G(F(\cdot_{1},\cdot_{2}), \eta_{2}) + \tr_{G}G([\cdot,\eta_{1}],[\cdot,\eta_{2}]) \\ &\ -\frac{1}{2}\tr_{G}G([\cdot_{1},\cdot_{2}],\eta_{1})G([\cdot_{1},\cdot_{2}],\eta_{2}) + \frac{1}{2}H^{2}(\eta_{1},\eta_{2})\\
  G\lp{\frac{\partial A}{\partial t}v, \eta} = &\ -G(D_{\cdot}F(v,\cdot), \eta) -D_{\cdot}G(F(v,\cdot),\eta) +\tr_{G}D_{v}G([\cdot,\eta],\cdot) + \frac{1}{2}H^{2}(v,\eta),\\
 \frac{\partial g}{ \partial t}(v_{1},v_{2}) = &\ -2\Rc_{g}(v_{1},v_{2}) +\frac{1}{2}D_{v_{1}}G(\cdot_{1},\cdot_{2})D_{v_{2}}G(\cdot_{1},\cdot_{2}) +G(F(v_{1},\cdot),F(v_{2},\cdot))+ \frac{1}{2}H^{2}(v_{1},v_{2})\\
 (\delta_{p}^{-1})^* \frac{\del \bar{H}}{\del t} =&\ dB ,\\
 B = &\ D_{\cdot}H(\cdot, *,*)  -D_{\cdot_{1}}G(\cdot_{2},\pi_{\mathfrak{G}}*) \wedge H(\cdot_{1},\cdot_{2},*)  \\& \   -\frac{1}{2}G(F(\cdot_{1},\cdot_{2}), \pi_{\mathfrak{G}}*) \wedge H(\cdot_{1},\cdot_{2},*)   +\frac{1}{2}\tr_{G}G([\cdot_{1},\cdot_{2}],\pi_{\mathfrak{G}}*) \wedge H(\cdot_{1},\cdot_{2},*).
\end{split}
\end{gather}
In the case when $\dim M=1$, we have $F \equiv 0$ and these equations become:
\begin{gather} 
\begin{split}\label{EQ}
 \frac{\partial G}{ \partial t}(\eta_{1},\eta_{2})= &\ (D_{\cdot}DG)_{\cdot}(\eta_{1},\eta_{2}) -D_{\cdot_{2}}G(\eta_{1},\cdot_{1})D_{\cdot_{2}}G(\cdot_{1},\eta_{2}) \\ &\   + \tr_{G}G([\cdot,\eta_{1}],[\cdot,\eta_{2}])  -\frac{1}{2}\tr_{G}G([\cdot_{1},\cdot_{2}],\eta_{1})G([\cdot_{1},\cdot_{2}],\eta_{2}) + \frac{1}{2}H^{2}(\eta_{1},\eta_{2}),\\
  G\lp{\frac{\partial A}{\partial t}v, \eta} = &\ \tr_{G}D_{v}G([\cdot,\eta],\cdot) + \frac{1}{2}H^{2}(v,\eta),\\
 \frac{\partial g}{ \partial t}(v_{1},v_{2}) =&\ \frac{1}{2}D_{v_{1}}G(\cdot_{1},\cdot_{2})D_{v_{2}}G(\cdot_{1},\cdot_{2}) + \frac{1}{2}H^{2}(v_{1},v_{2}),\\
 (\delta_{p}^{-1})^* \frac{\del \bar{H}}{\del t} =&\ dB,\\
 B = &\ D_{\cdot}H(\cdot, *,*)  -D_{\cdot_{1}}G(\cdot_{2},\pi_{\mathfrak{G}}*) \wedge H(\cdot_{1},\cdot_{2},*) +\frac{1}{2}\tr_{G}G([\cdot_{1},\cdot_{2}],\pi_{\mathfrak{G}}*) \wedge H(\cdot_{1},\cdot_{2},*).
\end{split}
\end{gather}

\section{Global existence and type III decay} \label{s:globex}

In this section, we will let $(\bar{g}(t),\bar{H}(t))$ be a time-dependent, invariant metric and closed three-form on a nilpotent  $\mathcal{G}$--principal bundle $P \rightarrow M$ that satisfy the flow equations (\ref{EQ0}). In Sections \ref{SECBRA12}-- \ref{SECDGH12}, we present our results regarding the evolution equations for $|[,]|^{2}$,  $|H^{\mathfrak{G}}|^{2}$, $|DG|^{2}$ and $\tr_{g}H^2$, where $H^{\mathfrak{G}}:=H(\pi_{\mathfrak{G}}*,\pi_{\mathfrak{G}}*,\pi_{\mathfrak{G}}*)$ and $\pi_{\mathfrak{G}}:\mathcal{E}\rightarrow \mathfrak{G}$ is the natural projection map. Based on these equations, we derive in Section \ref{SECGLOBAL12} our global existence and type III results for generalized Ricci flows.

\begin{rmk}
In the case when $\mathcal{G}$ is three dimensional, we express some of the equations below in terms of an orthonormal basis $\{\eta_{i} \}_{i=1,2,3}$ of $(\mathfrak{G}|_{m},G|_{m})$, where $m \in M$ and $\eta_{3} \in Z(\mathfrak{G})$. In doing so, we implicitly assume that $\mathcal{G}$ is nonabelian and hence the rank of $Z(\mathfrak{G})$ is one. These equations are still true in the abelian case on the condition that all terms with $|[,]|^{2}$ are set to zero.  
 \end{rmk}

\subsection{Evolution of \texorpdfstring{$|[,]|^{2}$}{}} \label{SECBRA12}

In this subsection, we derive the evolution equation for $\brs{[,]}^2$ first assuming the general setup of (\ref{EQ0}).  We then specialize to the case when $\mathcal{G}$ is three-dimensional and $M$ is   one-dimensional. This will lead to a sharp decay rate for $\brs{[,]}^2$ given in Section \ref{SECGLOBAL12}.

\begin{prop} \label{PROPbra1}Given a solution $(\bar{g}(t),\bar{H}(t))$  to the flow equations (\ref{EQ0}), we have  
\begin{align*}
\frac{\partial}{\partial t} \lo{[,]}^{2}&= \Delta |[,]|^{2}  
                   -|DG(\cdot,*_{1})G([*_{2},*_{3}],\cdot) -DG(*_{2}, \cdot)G([\cdot,*_{3}],*_{1}) + DG(*_
					     {3}, \cdot)G([\cdot,*_{2}],*_{1})  |^{2} 
        \\&+G(F(\cdot_{1},\cdot_{2}),\cdot_{3}) G(F(\cdot_{1},\cdot_{2}),\cdot_{4})G([\cdot_{3},\cdot_{5}], [\cdot_{4},\cdot_{5}])
        -\frac{1}{2}G(F(\cdot_{1},\cdot_{2}),[\cdot_{3},\cdot_{4}])G(F(\cdot_{1},\cdot_{2}),[\cdot_{3},\cdot_{4}])
       \\&       -2|G([*,\cdot], [*,\cdot]) -\frac{1}{2}G([\cdot_{1},\cdot_{2}],*)G([\cdot_{1},\cdot_{2}],*)|^{2} 
       -H^{2}(\cdot_{1},
                 \cdot_{2})G([\cdot_{1},\cdot],[\cdot_{2},\cdot]) 
             +\frac{1}{2}H^{2}([\cdot_{1},\cdot_{2}],[\cdot_{1},\cdot_{2}]).
\end{align*}
\end{prop}
\begin{proof}
We will outline the calculation involved in proving this evolution equation. First, using the flow equations (\ref{EQ0}), we obtain
\begin{gather*}
\begin{split}
\frac{\partial}{\partial t} \lo{[,]}^{2}=& \IP{ \frac{\partial G}{\partial t}, -2G([*,\cdot],[*,\cdot]) +G([\cdot_{1},\cdot_{2}],*)G([\cdot_{1},\cdot_{2}],*)}
							\\ =& -2(D_{\cdot}DG)_{\cdot}(\cdot_{1},\cdot_{2})G([\cdot_{1},\cdot_{3}], [\cdot_{2},\cdot_{3}])
							   +  (D_{\cdot}DG)_{\cdot}([\cdot_{1},\cdot_{2}],[\cdot_{1},\cdot_{2}]) 
			\\ & + 2D_{\cdot}G(\cdot_{1},\cdot_{3})D_{\cdot}G(\cdot_{2},\cdot_{3})G([\cdot_{1},\cdot_{4}],[\cdot_{2},\cdot_{4}])
					- D_{\cdot}G([\cdot_{1},\cdot_{2}],\cdot_{3})D_{\cdot}G([\cdot_{1},\cdot_{2}],\cdot_{3})    
		\\& 		-2|G([*,\cdot], [*,\cdot]) -\frac{1}{2}G([\cdot_{1},\cdot_{2}],*)G([\cdot_{1},\cdot_{2}],*)|^{2}	
					 \\&+G(F(\cdot_{1},\cdot_{2}),\cdot_{3}) G(F(\cdot_{1},\cdot_{2}),\cdot_{4})G([\cdot_{3},\cdot_{5}], [\cdot_{4},\cdot_{5}])
         -\frac{1}{2}G(F(\cdot_{1},\cdot_{2}),[\cdot_{3},\cdot_{4}])G(F(\cdot_{1},\cdot_{2}),[\cdot_{3},\cdot_{4}])
      \\&  -H^{2}(\cdot_{1},
                 \cdot_{2})G([\cdot_{1},\cdot],[\cdot_{2},\cdot]) 
             +\frac{1}{2}H^{2}([\cdot_{1},\cdot_{2}],[\cdot_{1},\cdot_{2}]).
\end{split}
\end{gather*}
Futhermore, a direct computation yields
\begin{gather*}
\begin{split}
\Delta |[,]|^{2} =&\ (D_{\cdot}DG)_{\cdot}([\cdot_{1},\cdot_{2}],[\cdot_{1},\cdot_{2}]) 
                          -4D_{\cdot}G([\cdot_{1},\cdot_{3}],[\cdot_{2},\cdot_{3}])D_{\cdot}G(\cdot_{1},\cdot_{2})    
                   \\&  -2(D_{\cdot}DG)_{\cdot}(\cdot_{1},\cdot_{2})G([\cdot_{1},\cdot_{3}],[\cdot_{2},\cdot_{3}])
   		+4G([\cdot_{1},\cdot_{3}],[\cdot_{2},\cdot_{3}])D_{\cdot}G(\cdot_{1},\cdot_{4})D_{\cdot}G(\cdot_{2},\cdot_{4})
  		\\& +2G([\cdot_{1},\cdot_{2}],[\cdot_{3},\cdot_{4}])D_{\cdot}G(\cdot_{1},\cdot_{3})D_{\cdot}G(\cdot_{2},\cdot_{4}).
\end{split}
\end{gather*}
Next, we observe the algebraic identity
\begin{gather*}
     \begin{split}
  |DG(\cdot,*_{1})& G([*_{2},*_{3}],\cdot) -DG(*_{2}, \cdot)G([\cdot,*_{3}],*_{1}) + DG(*_
					     {3}, \cdot)G([\cdot,*_{2}],*_{1})|^{2}\\
		=&\  - 4D_{\cdot}G([\cdot_{1},\cdot_{3}],[\cdot_{2},\cdot_{3}])D_{\cdot}G(\cdot_{1},\cdot_{2}) 
  					+2G([\cdot_{1},\cdot_{3}],[\cdot_{2},\cdot_{3}])D_{\cdot}G(\cdot_{1},\cdot_{4})D_{\cdot}G(\cdot_{2},\cdot_{4})\\ 
		&\ +2G([\cdot_{1},\cdot_{2}],[\cdot_{3},\cdot_{4}])D_{\cdot}G(\cdot_{1},\cdot_{3})D_{\cdot}G(\cdot_{2},\cdot_{4})
					+ D_{\cdot}G([\cdot_{1},\cdot_{2}],\cdot_{3})D_{\cdot}G([\cdot_{1},\cdot_{2}],\cdot_{3}).       
   \end{split}
 \end{gather*}
Combining the above three equations yields the result.
\end{proof}

\begin{prop} \label{PROPbra2} Let $(\bar{g}(t),\bar{H}(t))$ be a solution to the flow equations  (\ref{EQ0}) and suppose $\mathcal{G}$ is a three-dimensional, nilpotent Lie group and $M$ is one-dimensional.  The following holds true:
 \begin{align*}
\frac{\partial}{\partial t} \lo{[,]}^{2}=&\ \Delta |[,]|^{2} -\frac{3}{2} \lo{[,]}^{4} -S_{A},
\end{align*}
where 
\begin{align*}
S_{A} = |T|^{2}	+|H(\pi_{TM}*, \cdot_{1},\cdot_{2})[\cdot_{1},\cdot_{2}]|^{2}
			+ \frac{1}{6}\lo{H^{\mathfrak{G}}}^{2} |[,]|^{2}
\end{align*}
and 
\begin{align*}
T=DG(\cdot,*_{1})G([*_{2},*_{3}],\cdot) -DG(*_{2}, \cdot)G([\cdot,*_{3}],*_{1}) + DG(*_
					     {3}, \cdot)G([\cdot,*_{2}],*_{1}).		
\end{align*}
\begin{proof}The result follows by simplifying certain terms in Proposition \ref{PROPbra1} given the assumptions on $\mathcal{G}$ and $M$. First, the $F$--terms in that proposition vanish because $M$ is one dimensional. Next, we claim that 
\begin{gather} \label{EQAB}
\begin{split}
-2|G([*,\cdot], [*,\cdot]) -\frac{1}{2}G([\cdot_{1},\cdot_{2}],*)G([\cdot_{1},\cdot_{2}],*)|^{2}=  -\frac{3}{2}\lo{[,]}^{4}.
\end{split}
\end{gather}
To prove this, let $\{\eta_{i}\}_{i=1,2,3}$ be an orthonormal basis for $(\mathfrak{G}|_{m},G|_{m})$ such that $\eta_{3} \in Z(\mathfrak{G})$ and where $m \in M$.  If we set $Q=-2G([*,\cdot], [*,\cdot]) +G([\cdot_{1},\cdot_{2}],*)G([\cdot_{1},\cdot_{2}],*)$ then \[Q(\eta_{1},\eta_{1})=Q(\eta_{2},\eta_{2})=-Q(\eta_{3},\eta_{3})=-|[,]|^{2},\] while all other components are zero. $|Q|^{2}$ is then equal to $3|[,]|^{4}$, which proves (\ref{EQAB}).

Lastly, using the properties of $\{\eta_{i}\}_{i=1,2,3}$ and letting $v \in T_{m}M$ with $g(v,v)=1$, we have 
\begin{align*}
    -H^2(\cdot_{1},
                 \cdot_{2})G([\cdot_{1},\cdot],[\cdot_{2},\cdot]) &
             +\frac{1}{2}H^2([\cdot_{1},\cdot_{2}],[\cdot_{1},\cdot_{2}])\\
             =&\  -4H(v,\eta_{1},\eta_{2})^{2}|[\eta_{1},\eta_{2}]|^{2} -2H(\eta_{1},\eta_{2},\eta_{3})^{2}|[\eta_{1},\eta_{2}]|^{2}\\
             =&\ -|H(\pi_{TM}*, \cdot_{1},\cdot_{2})[\cdot_{1},\cdot_{2}]|^{2}
			- \frac{1}{6}\lo{H^{\mathfrak{G}}}^{2} |[,]|^{2}.
\end{align*}
\end{proof}
\end{prop}

\subsection{Evolution of \texorpdfstring{$|H^{\mathfrak{G}}|^{2}$}{}}

 \begin{prop} \label{PROPEVH} 
  Let $(\bar{g}(t),\bar{H}(t))$ be a solution to the flow equations  (\ref{EQ0}) and suppose $\mathcal{G}$ is a three-dimensional, nilpotent Lie group and $M$ is one-dimensional.  The following holds true:
  \begin{align*}
 \frac{\partial}{\partial t} \lo{H^{\mathfrak{G}}}^{2}= \Delta  |H^{\mathfrak{G}}|^{2}- \frac{1}{2} \lo{H^{\mathfrak{G}}}^{4} - |H^{\mathfrak{G}}|^{2}\left(\frac{1}{2} \lo{[,]}^{2}+ |DG(\cdot,\cdot)|^{2}+\tr_{g}H^2 \right).
 \end{align*}
 \begin{proof} First,  fix a basis $\{\eta_{i}\}_{i=1,2,3}$ for $\mathfrak{G}|_{m}$  such that $\eta_{3}\in Z(\mathfrak{G})$ and where $m \in M$. Using the flow equations (\ref{EQ}), it follows that 
 \begin{align*}
 \frac{\partial H}{\partial t}  (\eta_{1},\eta_{2},\eta_{3})
    &=dB(\eta_{1},\eta_{2},\eta_{3})
    \\&=-B([\eta_{1},\eta_{2}],\eta_{3}) +\mathrm{cyclic}(1,2,3)
    =-B([\eta_{1},\eta_{2}],\eta_{3})
    =0,    
 \end{align*}
 where we used Proposition \ref{PROPDB1} and that both $[\eta_{1},\eta_{2}]$ and $\eta_{3}$ lie in $Z(\mathfrak{G}|_{m}),$ which is one-dimensional when $\mathcal{G}$ is nonabelian.   It follows that
 \begin{align*}
\frac{\partial}{\partial t} |H^{\mathfrak{G}}|^{2} 
   					=&\ 2\frac{\partial H}{\partial t}(\eta_{i},\eta_{j},\eta_{k})H(\eta_{i},\eta_{j},\eta_{k})
					 -3H(\eta_{i},\eta_{k},\eta_{l})H(\eta_{j},\eta_{k},\eta_{l})\frac{\partial G}{\partial t}(\eta_{i},\eta_{j})
					\\ =&\ -6H(\eta_{1},\eta_{2},\eta_{3})^{2}\frac{\partial G}{\partial t}(\eta_{i},\eta_{i})\\
					 =&\ -|H^{\mathfrak{G}}|^{2} \left(D_{v}D_{v}G(\cdot,\cdot)-|DG|^{2} +\frac{1}{2}|[,]|^{2} +\frac{1}{2}H^2(\eta_{i},\eta_{i}) \right)\\
					 =&\ -|H^{\mathfrak{G}}|^{2} \left(D_{v}D_{v}G(\cdot,\cdot)-|DG|^{2} +\frac{1}{2}|[,]|^{2} + H^2(v,v)+ \frac{1}{2}|H^{\mathfrak{G}}|^{2} \right).
					 \end{align*}
To compute $\gD \brs{H^{\mathfrak G}}^2$, we first observe, using Proposition \ref{PROPdH} and $dH=0$,
 \begin{align}
  D|H^{\mathfrak{G}}|^{2}=&\ 
      2DH(\eta_{i},\eta_{j},\eta_{k})H(\eta_{i},\eta_{j},\eta_{k})
        -3H(\eta_{i},\eta_{k},\eta_{l})H(\eta_{j},\eta_{k},\eta_{l})DG(\eta_{i},\eta_{j}) \nonumber \\
        =&\ -6H(\eta_{1},\eta_{2},\eta_{3})^{2}DG(\cdot,\cdot) \nonumber  \\
        =&\ -|H^{\mathfrak{G}}|^{2}DG(\cdot,\cdot).\label{EQDH12}
 \end{align}
 Taking another derivative yields
 \begin{align*}
DD|H^{\mathfrak{G}}|^{2}
          =&\ D(-|H^{\mathfrak{G}}|^{2}DG(\cdot,\cdot))\\
          =&\ -D(|H^{\mathfrak{G}}|^{2})DG(\cdot,\cdot)
              -|H^{\mathfrak{G}}|^{2}D(DG(\cdot,\cdot) )\\
          =&\ |H^{\mathfrak{G}}|^{2}DG(\cdot_{1},\cdot_{1})DG(\cdot_{2},\cdot_{2})
             -|H^{\mathfrak{G}}|^{2}DDG(\cdot,\cdot)
             +|H^{\mathfrak{G}}|^{2}DG(\cdot_{1},\cdot_{2})DG(\cdot_{1},\cdot_{2}).
\end{align*}
Taking the trace and combining it with the above time derivative computation yields the result.
 \end{proof}
 \end{prop}

\subsection{Evolution of \texorpdfstring{$|DG|^{2}$}{}}

In this subsection, we assume $\mathcal{G}$ is a three-dimensional, nilpotent Lie group and $M$ is one-dimensional.  
 We will also let $v \in T_{m}M$ denote a $t$--dependent vector satisfying $g(v,v)=1$ and will let $\{\eta_{i}\}_{i=1,2,3}$ be an orthonormal basis for $(\mathfrak{G}|_{m},G|_{m})$ such that $\eta_{3} \in Z(\mathfrak{G})$. 
 
\begin{prop} \label{PROPDG1} Let $(\bar{g}(t),\bar{H}(t))$ be a solution to the flow equations  (\ref{EQ0}) and suppose $\mathcal{G}$ is a three-dimensional, nilpotent Lie group and $M$ is one-dimensional.  The following holds true:
 \begin{align*}
  \frac{\partial}{\partial t} \lo{DG}^{2}
             =&\ \Delta |DG|^{2} - \frac{1}{2}\lo{DG}^{4}
            -2|(D_{\cdot}DG)_{\cdot}(*_{1},*_{2})-D_{\cdot_{1}}G(*_{1},\cdot_{2})D_{\cdot_{1}}G(*_{2},\cdot_{2})|^{2}
            \\& -|[,]|^{2}(D_{v}G(\cdot,\cdot)-2D_{v}G(\eta_{3},\eta_{3}))^{2} 
            \\& -2|[,]|^{2}(D_{v}G(\cdot,\eta_{3})D_{v}G(\cdot,\eta_{3}) - |DG|^{2}_{Z(\mathfrak{G})} )
             -4| DG([\cdot,*],\cdot)|^{2}
             \\&-2H^{2}(v,\eta_{i})D_{v}G([\cdot,\eta_{i}],\cdot)
             -\frac{1}{2}\lo{DG}^{2}\tr_{g}H^{2}
             -2H(v,\cdot_{1},\cdot_{2})H(v,\cdot_{3},\cdot_{4})D_{v}G(\cdot_{1},\cdot_{3})D_{v}G(\cdot_{2},\cdot_{4})
         \\&    -2H(v,\cdot_{1},\cdot_{2})H(v,\cdot_{3},\cdot_{2})D_{v}G(\cdot_{1},\cdot_{4})D_{v}G(\cdot_{3},\cdot_{4})
                -\frac{1}{3}\lo{H^{ \mathfrak{G}}}^{2}|DG(\cdot,\cdot)|^{2}
    \\&+4D_{v}H(v,\cdot_{1},\cdot_{2})H(v, \cdot_{3},\cdot_{2})D_{v}G(\cdot_{1},\cdot_{3}).    
     \end{align*}
   \begin{proof}
 First note that for $w \in TM$ and $\eta \in \mathfrak{G}$, 
 \begin{align*}
  \frac{d}{dt} \lp{D_{w}\eta}= -\lb{\frac{\partial A}{\partial t}w,\eta}.
\end{align*}
Using this, a basic computation yields
 \begin{align*}
    \frac{\partial}{\partial t} |DG|^{2}
          =&\ 2D_{v}\frac{\partial G}{\partial t}(\cdot_{1},\cdot_{2})D_{v}G(\cdot_{1},\cdot_{2})
             +4G\lp{\lb{\frac{\partial A}{\partial t}v,\cdot_{1}},\cdot_{2}}D_{v}G(\cdot_{1},\cdot_{2})
             -|DG|^{2}\frac{\partial g}{\partial t}(\cdot,\cdot)
            \\&\ -2D_{v}G(\cdot_{1},\cdot_{2})D_{v}G(\cdot_{3},\cdot_{2})\frac{\partial G}{\partial t}(\cdot_{1},\cdot_{3}).
  \end{align*}
 Using the flow equations (\ref{EQ}), we compute each term above.  First
  \begin{gather*}
 \begin{split}
    D_{v}\frac{\partial G}{\partial t}(\cdot_{1},\cdot_{2})D_{v}G(\cdot_{1},\cdot_{2})
      =&\ D_{v}D_{v}D_{v}G(\cdot_{1},\cdot_{2})D_{v}G(\cdot_{1},\cdot_{2})
         -2D_{v}D_{v}G(\cdot_{1},\cdot_{2})D_{v}G(\cdot_{3},\cdot_{2})D_{v}G(\cdot_{1},\cdot_{3})
       \\&\ +D_{v}G(\cdot_{1},\cdot_{2})D_{v}G(\cdot_{2},\cdot_{3})D_{v}G(\cdot_{3},\cdot_{4})D_{v}G(\cdot_{4},\cdot_{1})
       +\frac{1}{2}\lo{[,]}^{2}|DG|^{2}
       \\&\ +D_{v}|[,]|^{2} \lp{ \frac{1}{2}D_{v}G(\cdot,\cdot)- D_{v}G(\eta_{3},\eta_{3})  }    
       \\&\ + |[,]|^{2}(-2D_{v}G(\cdot,\eta_{3}) D_{v}G(\cdot,\eta_{3}) + D_{v}G(\eta_{3},\eta_{3})D_{v}G(\eta_{3},\eta_{3})  )  
       \\&\ +\frac{1}{2}D_{v}H^{2}(\cdot_{1},\cdot_{2})D_{v}G(\cdot_{1},\cdot_{2}).
 \end{split}
 \end{gather*}
 Next we have
  \begin{gather*}
 \begin{split}
     G\lp{\lb{\frac{\partial A}{\partial t}v,\cdot_{1}},\cdot_{2}}D_{v}G(\cdot_{1},\cdot_{2})
           =-D_{v}G([\cdot_{1},\cdot_{2}],\cdot_{1})D_{v}G([\cdot_{3},\cdot_{2}],\cdot_{3})
           -\frac{1}{2}H^{2}(v, \eta_{i})D_{v}G([\cdot,\eta_{i}],\cdot).
  \end{split}
 \end{gather*}
 Next
 \begin{gather*}
 \begin{split}
   |DG|^{2}\frac{\partial g}{\partial t}(\cdot,\cdot)= \frac{1}{2}\lo{DG}^{4}+ \frac{1}{2}\lo{DG}^{2}H^{2}(v,v).
   \end{split}
 \end{gather*}
 Finally,
  \begin{gather*}
 \begin{split}
   D_{v}G(\cdot_{1},\cdot_{2})D_{v}G(\cdot_{3},\cdot_{2})\frac{\partial G}{ \partial t}(\cdot_{1},\cdot_{3})
   =&\  D_{v}G(\cdot_{1},\cdot_{2})D_{v}G(\cdot_{3},\cdot_{2})D_{v}D_{v}G(\cdot_{1},\cdot_{3}) 
    \\&\ -D_{v}G(\cdot_{1},\cdot_{2})D_{v}G(\cdot_{2},\cdot_{3})D_{v}G(\cdot_{3},\cdot_{4})D_{v}G(\cdot_{4},\cdot_{1})
    \\&\ +|[,]|^{2}\lp{\frac{1}{2}\lo{DG}^{2}   -D_{v}G(\eta_{3},\cdot)D_{v}G(\eta_{3},\cdot)   }
    \\&\ +\frac{1}{2}D_{v}G(\cdot_{1},\cdot_{2})D_{v}G(\cdot_{3},\cdot_{2})H^{2}(\cdot_{1},\cdot_{3}).
\end{split}
 \end{gather*}
We next compute $\gD \brs{DG}^2.$  We first observe
  \begin{gather}\label{EQDDG}
 \begin{split}
  D_{v}|DG|^{2}= 2(D_{v}D_{v}G(\cdot_{1},\cdot_{2})- D_{v}G(\cdot_{1},\cdot_{3})D_{v}G(\cdot_{2},\cdot_{3})   )D_{v}G(\cdot_{1},\cdot_{2}). 
 \end{split}
 \end{gather}
 Using this we obtain
 \begin{gather*}
 \begin{split}
    \Delta |DG|^{2}=&\
        2D_{v}D_{v}D_{v}G(\cdot_{1},\cdot_{2})D_{v}G(\cdot_{1},\cdot_{2})
        +2D_{v}D_{v}G(\cdot_{1},\cdot_{2})D_{v}D_{v}G(\cdot_{1},\cdot_{2})
        \\&\ -10D_{v}D_{v}G(\cdot_{1},\cdot_{2})D_{v}G(\cdot_{1},\cdot_{3})D_{v}G(\cdot_{2},\cdot_{3})
        \\&\ +6D_{v}G(\cdot_{1},\cdot_{2})D_{v}G(\cdot_{2},\cdot_{3})D_{v}G(\cdot_{3},\cdot_{4})D_{v}G(\cdot_{4},\cdot_{1}).
   \end{split}
 \end{gather*}
 Combining the above computations thus yields
 \begin{align*}
 \frac{\partial}{\partial t} \lo{DG}^{2}
             =&\ \Delta |DG|^{2} - \frac{1}{2} \lo{DG}^{4}
               -2D_{v}D_{v}G(\cdot_{1},\cdot_{2})D_{v}D_{v}G(\cdot_{1},\cdot_{2})
               +4D_{v}D_{v}G(\cdot_{1},\cdot_{2})D_{v}G(\cdot_{1},\cdot_{3})D_{v}G(\cdot_{2},\cdot_{3})
               \\&\ -2D_{v}G(\cdot_{1},\cdot_{2})D_{v}G(\cdot_{2},\cdot_{3})D_{v}G(\cdot_{3},\cdot_{4})D_{v}G(\cdot_{4},\cdot_{1})
               +D_{v}|[,]|^{2}(D_{v}G(\cdot,\cdot) -2D_{v}G(\eta_{3},\eta_{3}) )
               \\&\ +2|[,]|^{2}(-D_{v}G(\cdot,\eta_{3})D_{v}G(\cdot,\eta_{3}) +D_{v}G(\eta_{3},\eta_{3})D_{v}G(\eta_{3},\eta_{3})   )
               -4D_{v}G([\cdot_{1},\cdot_{2}],\cdot_{1})D_{v}G([\cdot_{3},\cdot_{2}],\cdot_{3})
               \\&\ +D_{v}H^{2}(\cdot_{1},\cdot_{2})D_{v}G(\cdot_{1},\cdot_{2})
               -\frac{1}{2} \lo{DG}^{2}H^{2}(v,v)
               -D_{v}G(\cdot_{1},\cdot_{2})D_{v}G(\cdot_{3},\cdot_{2})H^{2}(\cdot_{1},\cdot_{3})
               \\&\ -2H^{2}(v,\eta_{i})D_{v}G([\cdot,\eta_{i}],\cdot).   
\end{align*} 
It remains to perform some simplifications on these terms.  First, we have
\begin{gather*}
 \begin{split}
  -2|D_{v}D_{v}G- D_{v}G(*,\cdot)D_{v}G(*,\cdot)|^{2}=&\ 
    					-2D_{v}D_{v}G(\cdot_{1},\cdot_{2})D_{v}D_{v}G(\cdot_{1},\cdot_{2})
              \\&\ +4D_{v}D_{v}G(\cdot_{1},\cdot_{2})D_{v}G(\cdot_{1},\cdot_{3})D_{v}G(\cdot_{2},\cdot_{3})
               \\&\ -2D_{v}G(\cdot_{1},\cdot_{2})D_{v}G(\cdot_{2},\cdot_{3})D_{v}G(\cdot_{3},\cdot_{4})D_{v}G(\cdot_{4},\cdot_{1}).
 \end{split}
\end{gather*}
Furthermore, one computes that
\begin{gather} \label{EQDBRA}
 \begin{split}
  D|[,]|^{2}=-|[,]|^{2}(DG(\cdot,\cdot) -2DG(\eta_{3},\eta_{3}) )
  \end{split}
\end{gather}
and using this we obtain
 \begin{gather*}
 \begin{split}
     D_{v}|[,]|^{2}(D_{v}G(\cdot,\cdot) -2D_{v}G(\eta_{3},\eta_{3}) )= -|[,]|^{2}(D_{v}G(\cdot,\cdot) -2D_{v}G(\eta_{3},\eta_{3}) )^{2}.
 \end{split}
\end{gather*}
We also compute
\begin{gather*}
 \begin{split}
D_{v}H^{2}(\cdot_{1},\cdot_{2})D_{v}G(\cdot_{1},\cdot_{2})
           				=&\ 4D_{v}H(v,\cdot_{1},\cdot_{2})H(v,\cdot_{3},\cdot_{2})D_{v}G(\cdot_{1},\cdot_{3})
				  		\\&\ -2H(v,\cdot_{1},\cdot_{2})H(v,\cdot_{3},\cdot_{4})D_{v}G(\cdot_{1},\cdot_{3})D_{v}G(\cdot_{2},\cdot_{4})
				 \\&\ +2D_{v}H(\eta_{i},\eta_{k},\eta_{l})H(\eta_{j},\eta_{k},\eta_{l})D_{v}G(\eta_{i},\eta_{j})            
				    \\&\ -2H(\cdot_{1},\cdot_{2},\eta_{i})H(\cdot_{3},\cdot_{4},\eta_{i})D_{v}G(\cdot_{2},\cdot_{4})D_{v}G(\cdot_{1},\cdot_{3})\\
				    	=&\ 4D_{v}H(v,\cdot_{1},\cdot_{2})H(v,\cdot_{3},\cdot_{2})D_{v}G(\cdot_{1},\cdot_{3})
				  		\\&\ -2H(v,\cdot_{1},\cdot_{2})H(v,\cdot_{3},\cdot_{4})D_{v}G(\cdot_{1},\cdot_{3})D_{v}G(\cdot_{2},\cdot_{4})
				        \\&\ -2H(\cdot_{1},\cdot_{2},\eta_{i})H(\cdot_{3},\cdot_{4},\eta_{i})D_{v}G(\cdot_{2},\cdot_{4})D_{v}G(\cdot_{1},\cdot_{3}),
			\end{split}
\end{gather*}
where  in the last equality we used $DH(\eta_{1},\eta_{2},\eta_{3})=0$, which follows from  Proposition \ref{PROPdH} and the assumption that  $dH=0$. Furthermore, we have
 \begin{gather*}
 \begin{split}
					 D_{v}G(\cdot_{1},\cdot_{2})D_{v}G(\cdot_{3},\cdot_{2})H^{2}(\cdot_{1},\cdot_{3})
				=&\ 2H(v,\cdot_{1},\cdot_{2})H(v,\cdot_{3},\cdot_{2})D_{v}G(\cdot_{1},\cdot_{4})D_{v}G(\cdot_{3},\cdot_{4})
				\\&\ +H(\cdot_{1},\eta_{i},\eta_{j})H(\cdot_{2},\eta_{i},\eta_{j})D_{v}G(\cdot_{1},\cdot_{3})D_{v}G(\cdot_{2},\cdot_{3}).
			\end{split}
\end{gather*}
Lastly, we obtain
\begin{gather*}
 \begin{split}
-\frac{1}{3}\lo{H^{\mathfrak{G}}}^{2}|DG(\cdot,\cdot)|^{2}
    =&\ -2H(\cdot_{1},\cdot_{2},\eta_{i})H(\cdot_{3},\cdot_{4},\eta_{i})D_{v}G(\cdot_{2},\cdot_{4})D_{v}G(\cdot_{1},\cdot_{3}) 
    \\&\ -H(\cdot_{1},\eta_{i},\eta_{j})H(\cdot_{2},\eta_{i},\eta_{j})D_{v}G(\cdot_{1},\cdot_{3})D_{v}G(\cdot_{2},\cdot_{3}).
\end{split}
\end{gather*}
Combining the above computations yields the proposition.
 \end{proof}
 \end{prop}

\subsection{Evolution of \texorpdfstring{$\tr_{g}H^2$}{}}
In this subsection, we continue to assume $\mathcal{G}$ is a three-dimensional, nilpotent Lie group and $M$ is one-dimensional.  
 We will also let $v \in T_{m}M$ denote a $t$--dependent vector satisfying $g(v,v)=1$ and will let $\{\eta_{i}\}_{i=1,2,3}$ be an orthonormal basis for $(\mathfrak{G}|_{m},G|_{m})$ such that $\eta_{3} \in Z(\mathfrak{G})$. 

\begin{prop} \label{PROPHV1} 
Let $(\bar{g}(t),\bar{H}(t))$ be a solution to the flow equations  (\ref{EQ0}) and suppose $\mathcal{G}$ is a three-dimensional, nilpotent Lie group and $M$ is one-dimensional.  The following holds true:
\begin{align*}
 \frac{\partial}{\partial t}\tr_{g}H^2
 						=&\ \Delta \tr_{g}H^2 
						-2D_{v}H(v, \cdot_{1},\cdot_{2})D_{v}H(v, \cdot_{1},\cdot_{2})
						+4D_{v}H(v, \cdot_{1},\cdot_{2})H(v, \cdot_{3},\cdot_{2})D_{v}G(\cdot_{1},\cdot_{3})
				\\&\ -4H(v, \cdot_{1},\cdot_{2})H(v, \cdot_{3},\cdot_{2})D_{v}D_{v}G(\cdot_{1},\cdot_{3})
				+2H(v, \cdot_{1},\cdot_{2})H(v, \cdot_{3},\cdot_{2})D_{v}G(\cdot_{1},\cdot_{4})D_{v}G(\cdot_{3},\cdot_{4})
				\\&\ -2H(v, \cdot_{1},\cdot_{2})H(v, \cdot_{3},\cdot_{4})D_{v}G(\cdot_{1},\cdot_{3})D_{v}G(\cdot_{2},\cdot_{4})
				-\frac{1}{2}H^2(v,v)|DG|^{2}
				\\&\ -2D_{v}G([\cdot,\eta_{i}],\cdot)H^2(v,\eta_{i})
				-2|H(\pi_{TM}*,\cdot_{1},\cdot_{2})[\cdot_{1},\cdot_{2}]|^{2}
				\\&\ -2H(v,\cdot_{1},\cdot_{2})H(v,\cdot_{3},\cdot_{2})H(v,\cdot_{1},\cdot_{4})H(v,\cdot_{3},\cdot_{4})                        	             -2H^2(v,\eta_{i})H^2(v,\eta_{i})
				   -\frac{1}{2}H^2(v,v)H^2(v,v).
				\end{align*}
\begin{proof}
We begin with the basic calculation
\begin{align*}
\frac{\partial}{\partial t}\tr_{g}H^{2}
 							=&\ 2\frac{\partial H}{\partial t}(v, \cdot_{1},\cdot_{2})H(v, \cdot_{1},\cdot_{2})
							   -H(v, \cdot_{1},\cdot_{2})H(v, \cdot_{1},\cdot_{2})\frac{\partial g}{\partial t}(v,v)
							  \\&\ -2H(v, \cdot_{1},\cdot_{2})H(v, \cdot_{3},\cdot_{2})\frac{\partial G}{\partial t}(\cdot_{1},\cdot_{3}).
\end{align*}
Using (\ref{EQ}) and Proposition \ref{PROPDB1}, we first compute
	\begin{gather*}
	\begin{split}
	 \frac{\partial H}{\partial t}& (v, \cdot_{1},\cdot_{2})H(v, \cdot_{1},\cdot_{2})\\
	      =&\ D_{v}D_{v}H(v,\cdot_{1},\cdot_{2})H(v,\cdot_{1},\cdot_{2})
	         -2D_{v}D_{v}G(\cdot_{1},\cdot_{2})H(v,\cdot_{1},\cdot_{3})H(v,\cdot_{2},\cdot_{3})
			\\&\ -2D_{v}G(\cdot_{1},\cdot_{2})D_{v}H(v,\cdot_{1},\cdot_{3})H(v,\cdot_{2},\cdot_{3})
				 \\&\ +2D_{v}G(\cdot_{1},\cdot_{2})D_{v}G(\cdot_{1},\cdot_{3})H(v,\cdot_{3},\cdot_{4})H(v,\cdot_{2},\cdot_{4})
				\\&\  +\frac{1}{2}D_{v}(G([\cdot_{1},\cdot_{2}], \pi_{\mathfrak{G}}*) \wedge H(\cdot_{1},\cdot_{2}, *))(\cdot_{3},\cdot_{4}) H(v, \cdot_{3},\cdot_{4})
		 \\&\ -\frac{1}{2}G([\cdot_{1},\cdot_{2}], [\cdot_{3},\cdot_{4}])H(v,\cdot_{1},\cdot_{2}) H(v,\cdot_{3},\cdot_{4}) 
		 -D_{v}G([\cdot_{1},\eta_{i}],\cdot_{1})H(\eta_{i},\cdot_{2},\cdot_{3})H(v,\cdot_{2},\cdot_{3})
		 \\&\ -\frac{1}{2}H^2(v,\eta_{k})H(\eta_{k},\cdot_{1},\cdot_{2})H(v,\cdot_{1},\cdot_{2}).
	\end{split}
\end{gather*}
We also have
	\begin{gather*}
	\begin{split}
	 D_{v} & (G([\cdot_{1},\cdot_{2}], \pi_{\mathfrak{G}}*) \wedge H(\cdot_{1},\cdot_{2}, *))(\cdot_{3},\cdot_{4}) H(v, \cdot_{3},\cdot_{4})\\
	 		=&\ 2D_{v}G([\eta_{k},\eta_{l}],\cdot_{1})H(\eta_{k},\eta_{l},\cdot_{2})H(v,\cdot_{1},\cdot_{2}) +2D_{v}H(\eta_{i},\eta_{j},\eta_{k})H(v,[\eta_{i},\eta_{j}],\eta_{k})
			\\&\ -4D_{v}G(\cdot_{1},\cdot_{2})H(\cdot_{2},\eta_{i},\cdot_{3})H(v,[\cdot_{1},\eta_{i}],\cdot_{3})\\
	=&\ 2D_{v}G([\eta_{k},\eta_{l}],\cdot_{1})H(\eta_{k},\eta_{l},\cdot_{2})H(v,\cdot_{1},\cdot_{2})  -4D_{v}G(\cdot_{1},\cdot_{2})H(\cdot_{2},\eta_{i},\cdot_{3})H(v,[\cdot_{1},\eta_{i}],\cdot_{3}),
		\end{split}
\end{gather*}
where  in the last equality we used $DH(\eta_{1},\eta_{2},\eta_{3})=0$, which follows from  Proposition \ref{PROPdH} and the assumption that  $dH=0$.  Furthermore, we observe
 \begin{gather*}
	\begin{split}
 H(v, \cdot_{1},\cdot_{2})H(v, \cdot_{1},\cdot_{2})\frac{\partial g}{\partial t}(v,v)
   = \frac{1}{2}H^2(v,v)|DG|^{2}
      +\frac{1}{2}H^2(v,v)H^2(v,v).
\end{split}
\end{gather*}
We also compute
\begin{gather*}
	\begin{split}
 H(v, \cdot_{1},\cdot_{2})& H(v, \cdot_{3},\cdot_{2})\frac{\partial G}{\partial t}(\cdot_{1},\cdot_{3})\\
   						=&\ H(v, \cdot_{1},\cdot_{2})H(v, \cdot_{3},\cdot_{2})D_{v}D_{v}G(\cdot_{1},\cdot_{3}) -H(v, \cdot_{1},\cdot_{2})H(v, \cdot_{3},\cdot_{2})D_{v}G(\cdot_{1},\cdot_{4})D_{v}G(\cdot_{3},\cdot_{4})
		\\&\ +H(v, \cdot_{1},\cdot_{2})H(v, \cdot_{3},\cdot_{2})\lp{G([\cdot_{1},\cdot_{4}], [\cdot_{3},\cdot_{4}]) 
		        -\frac{1}{2}G([\cdot_{4},\cdot_{5}],\cdot_{1}) G([\cdot_{4},\cdot_{5}],\cdot_{3})     }
	\\&\	+\frac{1}{2}H(v, \cdot_{1},\cdot_{2})H(v, \cdot_{3},\cdot_{2}) H^2(\cdot_{1},\cdot_{3}).       
\end{split}
\end{gather*}
We turn to computing $\gD \tr_g H^2$.  First, we observe
\begin{gather*}
\begin{split}
 D\tr_{g}H^2
            = 2DH(v,\cdot_{1},\cdot_{2})H(v,\cdot_{1},\cdot_{2})
               -2H(v,\cdot_{1},\cdot_{2})H(v,\cdot_{3},\cdot_{2})DG(\cdot_{1},\cdot_{3}).
\end{split}
\end{gather*}
Using this, we compute
\begin{gather*}
\begin{split}
 \Delta \tr_{g}H^2 =&\
          2D_{v}D_{v}H(v, \cdot_{1},\cdot_{2})H(v, \cdot_{1},\cdot_{2})
          +2D_{v}H(v, \cdot_{1},\cdot_{2})D_{v}H(v, \cdot_{1},\cdot_{2})
         \\&\ -8D_{v}H(v,\cdot_{1},\cdot_{2})H(v,\cdot_{3},\cdot_{2})D_{v}G(\cdot_{1},\cdot_{3})
          -2H(v,\cdot_{1},\cdot_{2})H(v,\cdot_{3},\cdot_{2})D_{v}D_{v}G(\cdot_{1},\cdot_{3})
      \\&\ +4H(v,\cdot_{1},\cdot_{2})H(v,\cdot_{3},\cdot_{2})D_{v}G(\cdot_{1},\cdot_{4})D_{v}G(\cdot_{3},\cdot_{4})\\
       &\ +2 H(v,\cdot_{1},\cdot_{2})H(v,\cdot_{3},\cdot_{4})D_{v}G(\cdot_{1},\cdot_{3})D_{v}G(\cdot_{2},\cdot_{4}).
      \end{split}
\end{gather*}
Combining the above computations yields the preliminary formula
\begin{align*}
 \frac{\partial}{\partial t}\tr_{g}H^2
 						=&\ \Delta \tr_{g}H^2 
						-2D_{v}H(v, \cdot_{1},\cdot_{2})D_{v}H(v, \cdot_{1},\cdot_{2})
						+4D_{v}H(v, \cdot_{1},\cdot_{2})H(v, \cdot_{3},\cdot_{2})D_{v}G(\cdot_{1},\cdot_{3})
				\\&\ -4H(v, \cdot_{1},\cdot_{2})H(v, \cdot_{3},\cdot_{2})D_{v}D_{v}G(\cdot_{1},\cdot_{3})
				+2H(v, \cdot_{1},\cdot_{2})H(v, \cdot_{3},\cdot_{2})D_{v}G(\cdot_{1},\cdot_{4})D_{v}G(\cdot_{3},\cdot_{4})
				\\&\ -2H(v, \cdot_{1},\cdot_{2})H(v, \cdot_{3},\cdot_{4})D_{v}G(\cdot_{1},\cdot_{3})D_{v}G(\cdot_{2},\cdot_{4})
				-\frac{1}{2}H^2(v,v)|DG|^{2}
				\\&\ 
	            +2D_{v}G([\eta_{i},\eta_{j}],\cdot_{1})H(\eta_{i},\eta_{j},\cdot_{2})H(v,\cdot_{1},\cdot_{2})
				 -2D_{v}G([\cdot_{1},\eta_{i}],\cdot_{1})H(\eta_{i},\cdot_{2},\cdot_{3})H(v,\cdot_{2},\cdot_{3})
			\\&\	-4D_{v}G(\eta_{i},\eta_{j})H(\eta_{j},\eta_{k},\eta_{l})H(v,[\eta_{i},\eta_{k}],\eta_{l})
			     -G([\cdot_{1},\cdot_{2}], [\cdot_{3},\cdot_{4}])H(v,\cdot_{1},\cdot_{2})H(v,\cdot_{3},\cdot_{4})
				\\&\ -2|[,]|^{2}H(v,\cdot_{1},\cdot_{2})H(v,\cdot_{3},\cdot_{2})\lp{\frac{1}{2}G(\cdot_{1},\cdot_{4})G(\cdot_{3},\cdot_{4}) -G(\cdot_{1},\eta_{3})G(\cdot_{3},\eta_{3}) }
				\\&\ -H(v,\cdot_{1},\cdot_{2})H(v,\cdot_{3},\cdot_{2})H^2(\cdot_{1},\cdot_{3})
				-H^2(v,\eta_{i})H^2(v,\eta_{i})
				-\frac{1}{2}H^2(v,v)H^2(v,v).
				\end{align*}
We now observe various simplifications.  First,
\begin{gather*}
\begin{split}
2D_{v}G([\eta_{i},\eta_{j}],\cdot_{1})H(\eta_{i},\eta_{j},\cdot_{2})H(v,\cdot_{1},\cdot_{2})-4D_{v}G(\eta_{i},\eta_{j})H(\eta_{j},\eta_{k},\eta_{l})H(v,[\eta_{i},\eta_{k}],\eta_{l})=0.
\end{split}
\end{gather*}
Next, we have
\begin{gather*}
\begin{split}
2|H(\pi_{TM}*,\cdot_{1},\cdot_{2})[\cdot_{1},\cdot_{2}]|^{2} =&\ G([\cdot_{1},\cdot_{2}], [\cdot_{3},\cdot_{4}])H(v,\cdot_{1},\cdot_{2})H(v,\cdot_{3},\cdot_{4})
				\\&\ +2|[,]|^{2}H(v,\cdot_{1},\cdot_{2})H(v,\cdot_{3},\cdot_{2}) \left(\frac{1}{2}G(\cdot_{1},\cdot_{4})G(\cdot_{3},\cdot_{4}) -G(\cdot_{1},\eta_{3})G(\cdot_{3},\eta_{3}) \right).
\end{split}
\end{gather*}
Lastly, we have
\begin{gather*}
\begin{split}
 -H(v,\cdot_{1},\cdot_{2})H(v,\cdot_{3},\cdot_{2})H^2(\cdot_{1},\cdot_{3})
                = -2H(v,\cdot_{1},\cdot_{2})H(v,\cdot_{3},\cdot_{2})H(v,\cdot_{1},\cdot_{4})H(v,\cdot_{3},\cdot_{4})                        	             -H^2(v,\eta_{i})H^2(v,\eta_{i}).
\end{split}
\end{gather*}
Combining the above computations yields the proposition.
\end{proof}
\end{prop}

\subsection{Evolution of \texorpdfstring{$\brs{DG}^2 + \tr_{g}H^2$}{}} \label{SECDGH12}

In this subsection, we continue to assume $\mathcal{G}$ is a three-dimensional, nilpotent Lie group and $M$ is one-dimensional.  
 We will also let $v \in T_{m}M$ denote a $t$--dependent vector satisfying $g(v,v)=1$ and will let $\{\eta_{i}\}_{i=1,2,3}$ be an orthonormal basis for $(\mathfrak{G}|_{m},G|_{m})$ such that $\eta_{3} \in Z(\mathfrak{G})$. 

\begin{prop} \label{PROPEVDG} 
 Let $(\bar{g}(t),\bar{H}(t))$ be a solution to the flow equations  (\ref{EQ0}) and suppose $\mathcal{G}$ is a three-dimensional, nilpotent Lie group and $M$ is one-dimensional.  The following holds true:
 \begin{align*}
 \frac{\partial}{\partial t} \lp{|DG|^{2}+ \tr_{g}H^{2}}= \Delta (|DG|^{2}+ \tr_{g}H^{2}) -\frac{1}{2} \lp{|DG|^{2}+ \tr_{g}H^{2}}^{2} -S_{B},
 \end{align*}
 where 
 \begin{align*}
 S_{B}= &\ 2|(D_{\cdot}DG)_{\cdot}(*_{1},*_{2})-D_{\cdot_{1}}G(*_{1},\cdot_{2})D_{\cdot_{1}}G(*_{2},\cdot_{2}) +H(v,*_{1},\cdot)H(v,*_{2},\cdot)|^{2}
            \\&\  +2|D_{\cdot}H(\cdot,*_{1},*_{2})+H(v,\cdot,*_{1})D_{v}G(\cdot,*_{2}) -H(v,\cdot,*_{2})D_{v}G(\cdot,*_{1})|^{2}
            \\&\ +|[,]|^{2}( D_{v}G(\cdot,\cdot)-2D_{v}G(\eta_{3},\eta_{3}) )^{2}
            + 2|[,]|^{2}(D_{v}G(\cdot,\eta_{3}) D_{v}G(\cdot,\eta_{3}) - |DG|_{Z(\mathfrak{G})}^{2}  )
            \\&\ +2|DG([\cdot,*],\cdot)|^{2} +2|D_{v}G([\cdot,*],\cdot) + H^{2}(v, \pi_{\mathfrak{G}}*)|^{2}
            \\&\ +2|H(\pi_{TM}*,\cdot_{1},\cdot_{2})[\cdot_{1},\cdot_{2}]|^{2} +\frac{1}{3} \lo{H^{\mathfrak{G}}}^{2}|DG(\cdot,\cdot)|^{2}.
 \end{align*}
\begin{proof} This follows directly from Propositions \ref{PROPDG1} and \ref{PROPHV1} after observing the further simplification
 \begin{gather*}
 \begin{split}
 2|(D_{\cdot}&DG)_{\cdot}(*_{1},*_{2})-D_{\cdot_{1}}G(*_{1},\cdot_{2})D_{\cdot_{1}}G(*_{2},\cdot_{2}) +H(v,*_{1},\cdot)H(v,*_{2},\cdot)|^{2}
   \\& \         +2|D_{\cdot}H(\cdot,*_{1},*_{2})+H(v,\cdot,*_{1})D_{v}G(\cdot,*_{2}) -H(v,\cdot,*_{2})D_{v}G(\cdot,*_{1})|^{2} 
    \\ =&\   2 |(D_{\cdot}DG)_{\cdot}(*_{1},*_{2})-D_{\cdot_{1}}G(*_{1},\cdot_{2})D_{\cdot_{1}}G(*_{2},\cdot_{2})|^{2}
  +4H(v, \cdot_{1},\cdot_{2})H(v, \cdot_{3},\cdot_{2})D_{v}D_{v}G(\cdot_{1},\cdot_{3})
   \\&\  +2H(v,\cdot_{1},\cdot_{2})H(v,\cdot_{3},\cdot_{2})H(v,\cdot_{1},\cdot_{4})H(v,\cdot_{3},\cdot_{4})   
    +2D_{v}H(v, \cdot_{1},\cdot_{2})D_{v}H(v, \cdot_{1},\cdot_{2})
   \\&\  -8D_{v}H(v, \cdot_{1},\cdot_{2})H(v, \cdot_{3},\cdot_{2})D_{v}G(\cdot_{1},\cdot_{3})
    +4H(v, \cdot_{1},\cdot_{2})H(v, \cdot_{3},\cdot_{4})D_{v}G(\cdot_{1},\cdot_{3})D_{v}G(\cdot_{2},\cdot_{4}).
\end{split}
 \end{gather*}
 \end{proof}
 \end{prop}

\subsection{Global existence and universal type III bounds}\label{SECGLOBAL12}

\begin{cor} \label{c:typeIIIestimates} 
Let $(\bar{g}(t),\bar{H}(t))$ be a solution to the flow equations  (\ref{EQ0}) and suppose $\mathcal{G}$ is a three-dimensional, nilpotent Lie group and $M$ is one-dimensional and compact.
Then for all smooth existence times $t > 0$ one has
\begin{align*}
\brs{[,]}^2 \leq&\ \frac{2}{3t}, \qquad
|H^{\mathfrak{G}}|^{2} \leq \frac{2}{t}, \qquad 
|DG|^{2}+ \tr_{g}H^2 \leq \frac{2}{t}.
\end{align*}
\begin{proof} These are straightforward applicaions of the maximum principle using Propositions \ref{PROPbra2}, \ref{PROPEVH}, and \ref{PROPEVDG}.  We give the proof for $\brs{[,]}^2$, with the remaining cases being analogous.  Let
\begin{align*}
    \Phi(x,t) := t \brs{[,]}^2.
\end{align*}
A simple computation using Proposition \ref{PROPbra2} gives
\begin{align*}
    \left(\dt - \gD \right) \Phi \leq&\ \brs{[,]}^2 \left( 1 - \frac{3}{2} \Phi \right).
\end{align*}
By the maximum principle, since $\sup_{M \times \{0\}} \Phi = 0 \leq \frac{2}{3}$, it follows that $\sup_{M \times \{t\}} \Phi \leq \frac{2}{3}$ for all smooth existence times $t$, yielding the result.
\end{proof}
\end{cor}

\begin{prop} \label{p:curvatureev} Given a solution to generalized Ricci flow as above, suppose further that $\mathcal{G}$ is a three-dimensional nilpotent Lie group and $M$ is one-dimensional.  Then
\begin{align*}
    \dt D D G =&\ \gD D D G + D^3 G \star D G + \left(D H + D^2 G + \left( H + DG + [,] \right)^{*2} \right)^{*2},\\
    \dt D H =&\ \gD D H + D^2 H \star H + \left(D H + D^2 G + \left( H + DG + [,] \right)^{*2} \right)^{*2}.
\end{align*}
Also one has the universal inequalities
\begin{align*}
    \left(\dt - \gD \right) \left( \brs{D^2 G}^2 + \brs{D H}^2 \right) \leq&\ - \left( \brs{D^3 G}^2 + \brs{D^2 H}^2 \right) + C_1 \left( \brs{D^2 G}^3 + \brs{D H}^3 \right) + C_2 t^{-3}\\
    \left(\dt - \gD \right) \left( \brs{DG}^2 + \tr_g H^2 \right) \leq&\ - \brs{D^2 G}^2 - \brs{D H}^2 + C t^{-2}.
\end{align*}
\begin{proof} The evolution equations for $DDG$ and $DH$ are straightforward consequences of (\ref{EQ}).  Building upon these, the differential inequality for $\brs{D^2G}^2 + \brs{D H}^2$ follows from further elementary estimates and Corollary \ref{c:typeIIIestimates}.  To prove the final inequality we first notice using Proposition \ref{PROPdH} and Corollary \ref{c:typeIIIestimates} that
\begin{align*}
    S_B \geq&\ 2 \brs{D D G + DG^{*2} + H^{*2}}^2 + 2 \brs{DH + H \star DG}^2\\
    \geq&\ \brs{D DG}^2 + \brs{DH}^2 - C t^{-2},
\end{align*}
and the claim then follows from Proposition \ref{PROPEVDG}.
\end{proof}
\end{prop}

\begin{thm} \label{t:lte3+1} Let $(P^4, \bar{g}, \bar{H})$ be a Riemannian metric and closed three-form on the nilpotent $\mathcal G$-principal bundle $P \to M$, where $M$ is compact and one-dimensional.  Then the solution to generalized Ricci flow with initial data $(\bar{g}, \bar{H})$ exists on $[0,\infty)$, and there is a universal constant $C$ so that the solution satisfies the type III estimates
\begin{align*}
    \sup_{M \times [0,\infty)} t \left( \brs{\bar{\Rm}} + \brs{\bar{H}}^2 \right) \leq C,\\
    \sup_{M \times [0,\infty)} t \brs{[,]}^2 \leq \frac{2}{3},\\
    \sup_{M \times [0,\infty)} t \brs{H^{\mathfrak G}}^2 \leq 2,\\
    \sup_{M \times [0,\infty)} t \left( \brs{DG}^2 + \tr_g H^2 \right) \leq 2.
\end{align*}
\begin{proof} By (\cite{Streetsexpent}, cf. \cite{garciafernStreets2020} Theorem 5.23) to show the long time existence it suffices to establish a uniform bound on the curvature of $\bar{g}_t$, therefore it suffices to show the type III estimates for any smooth existence interval.  Note that in Corollary \ref{c:typeIIIestimates} we already showed the required estimates on the various first order quantities.  Furthermore, using the expression for the curvature of an invariant metric in this setting (cf. \cite{GindiStreets} Theorem 2.12), it suffices to estimate $\brs{D^2 G}^2$.  To obtain this estimate we fix a large constant $A > 0$ to be determined and let
\begin{align*}
    \Phi = t^2 \left( \brs{D^2 G}^2 + \brs{D H}^2 \right) + t A \left( \brs{DG}^2 + \tr_{g}H^2 \right).
\end{align*}
Using Proposition \ref{p:curvatureev} we obtain
\begin{align*}
    \left(\dt - \gD \right) \Phi \leq&\ 2 t \left( \brs{D^2 G}^2 + \brs{D H}^2 \right) + t^2 C_1 \left( \brs{D^2 G}^3 + \brs{DH}^3 \right)\\
    &\ - t A \left( \brs{D^2 G}^2 + \brs{D H}^2 \right) + C_4 (1 + A) t^{-1}\\
    \leq&\ t \left( \brs{D^2 G}^2 +\brs{DH}^2 \right) \left( 2 - A + tC_1 \left( \brs{D^2G} + \brs{DH} \right) \right) + C_4 (1 + A) t^{-1}.
\end{align*}
We claim that for $A$ chosen sufficiently large we have
\begin{align*}
    \sup_{M \times [0,T]} \Phi \leq A^{\frac{3}{2}}
\end{align*}
for any smooth existence time $T$.  Suppose there exists a first spacetime point $(x_0, t_0)$ where $\Phi = A^{\frac{3}{2}}$.  At this point we obtain by the maximum principle, the estimates of Corollary \ref{c:typeIIIestimates}, and applications of the arithmetic-geometric mean,
\begin{align*}
    0 \leq&\ t \left( \brs{D^2 G}^2 +\brs{DH}^2 \right) \left( 2 - A + t C_1 \left( \brs{D^2G} + \brs{DH} \right) \right) + C_4 (1 + A) t^{-1}\\
    =&\ t^{-1} \left( \Phi - t A (\brs{DG}^2 + \tr_g H^2 ) \right) \left( 2 - A + C_1 \left( \Phi - t A( \brs{DG}^2 + \tr_g H^2 \right)^{\frac{1}{2}} \right) + C_4 (1 + A) t^{-1}\\
    \leq&\ t^{-1} \left( A^{\frac{3}{2}} - C A \right) \left( -\frac{A}{2} + C_5 \right) + C_4 (1 + A) t^{-1}\\
    \leq&\ t^{-1} \left( - A^{\frac{5}{2}}{4} + C_6 \right)\\
    <&\ 0,
\end{align*}
where the last line follows by choosing $A$ sufficiently large.  This is a contradiction, and so the estimate for $\brs{\bar{\Rm}}$ follows.  For the diameter estimate, we note that by the evolution equations (\ref{EQ}), we can express $g_t = e^{u_t} g_0$, where
\begin{align*}
    \frac{\del u}{\del t} = \frac{1}{2} \left( \brs{DG}^2 + \tr_g H^2 \right) \leq t^{-1}.
\end{align*}
By a straightforward integration in time, we obtain the upper bound $g_t \leq t g_0$, and the claimed diameter estimate follows.
\end{proof}
\end{thm}

\section{Convergence} \label{s:conv}
\subsection{Monotone energy} \label{SECFUNC}

We will now introduce a functional that is monotone along the generalized Ricci flow equations (\ref{EQ}). We will use this in Section \ref{SECCONV} to derive rigidity results for blowdown limits. 

\begin{defn} \label{d:energy}
Let $\tau \in (0,\infty)$ and $(\bar{g}, \bar{H})$ be an invariant metric and three-form on a $\mathcal{G}$--principal bundle $P \rightarrow M$, where $M$ is one-dimensional and compact.
Define 
\begin{align*}
\mathcal{I}(\bar{g}, \bar{H}, \tau)= \tau \int_{M}\lp{|DG|^{2} + \tr_{g}H^2 +\frac{2}{\tau} }\frac{dV_{g}}{\sqrt{\tau}}.
\end{align*}
\end{defn}

It is elementary to observe that $\mathcal{I}$ is invariant under principal bundle isomorphisms and is scale invariant: $\mathcal{I}(s\bar{g}, s\bar{H}, s\tau)= \mathcal{I}(\bar{g}, \bar{H}, \tau)$, for all $s>0$.  We next show that $\mathcal I$ is monotone along generalized Ricci flow.

\begin{prop} \label{p:Imonotone}
Let $(\bar{g}(t), \bar{H}(t))$ be a time-dependent invariant metric and closed three-form satisfying (\ref{EQ}) on a $\mathcal{G}$--principal bundle $P \rightarrow M$, where $\mathcal{G}$ is a three dimensional, nilpotent Lie group and $M$ is one-dimensional and compact. Then 
\begin{align*}
 \frac{d}{dt}\mathcal{I}(\bar{g}(t), \bar{H}(t),t)= -t\int_{M}S_{B}\frac{dV_{g}}{\sqrt{t}} -\frac{t}{4}\int_{M}\lp{|DG|^{2} +\tr_{g}H^2-\frac{2}{t}}^{2}\frac{dV_{g}}{\sqrt{t}} \leq 0.
  \end{align*}
\end{prop}

\begin{proof}
Setting $Q:=|DG|^{2} + \tr_{g}H^2$, first note  by (\ref{EQ}), 

 \begin{align*}
\frac{d}{dt} \frac{dV_{g}}{\sqrt{t}}
 = \lp{\frac{1}{2}\frac{\partial g}{\partial t}(\cdot,\cdot) -\frac{1}{2t}}\frac{dV_{g}}{\sqrt{t}}
 =\lp{\frac{1}{4}Q-\frac{1}{2t}}\frac{dV_{g}}{\sqrt{t}}.
\end{align*}
Using this together with Proposition \ref{PROPEVDG}, yields
\begin{align*}
\frac{d}{dt} \lp{t\int_{M}Q\frac{dV_{g}}{\sqrt{t}}}
       & =\int_{M} \lp{Q +t \frac{dQ}{dt}} \frac{dV_{g}}{\sqrt{t}} +t\int_{M} Q \frac{d}{dt}\lp{  \frac{dV_{g}}{\sqrt{t}} }
       \\&= \int_{M} \left( -t S_{B} + Q -\frac{t}{2} Q^{2}  +t Q\lp{ \frac{1}{4}Q -\frac{1}{2t}}  \right) \frac{dV_{g}}{\sqrt{t}}
  \\& =  -t\int_{M} S_{B}\frac{dV_{g}}{\sqrt{t}} -t\int_{M}Q \left(\frac{1}{4}Q -\frac{1}{2t} \right) \frac{dV_{g}}{\sqrt{t}}.
      \end{align*}
      
     Hence

\begin{align*}
 \frac{d}{dt}\mathcal{I}(\bar{g}(t), \bar{H}(t),t)&=
\frac{d}{dt} \lp{\int_{M} \lp{t Q  + 2} \frac{dV_{g}}{\sqrt{t}}}
   \\&=  -t\int_{M} S_{B}\frac{dV_{g}}{\sqrt{t}} 
    -\frac{t}{4}\int_{M} \left( Q^{2} -\frac{2}{t}Q \right) \frac{dV_{g}}{\sqrt{t}} 
    +2\int_{M}\left(\frac{1}{4}Q -\frac{1}{2t} \right)  \frac{dV_{g}}{\sqrt{t}}
    \\&=  -t\int_{M} S_{B}\frac{dV_{g}}{\sqrt{t}} 
         -\frac{t}{4}\int_{M} \lp{Q-\frac{2}{t}}^{2} \frac{dV_{g}}{\sqrt{t}}.
        \end{align*}
   \end{proof}

With this monotonicity in place, it is natural to understand the case when $\mathcal I$ is fixed.  The next proposition shows that such metrics are quite rigid.

\begin{prop} \label{PROPSB1}Let $(\bar{g}, \bar{H})$ be an invariant metric and closed three-form on a $\mathcal{G}$--principal bundle $P \rightarrow M$, where $\mathcal{G}$ is a three dimensional, nilpotent Lie group and $M$ is one-dimensional and compact. If $S_{B}=0$ then the following holds true. 
   \begin{enumerate}
  \item[1)]  Using the $G_{t}$--orthogonal  splitting $\mathfrak{G}= Z(\mathfrak{G})^{\perp} \oplus Z(\mathfrak{G}),$ 
   $D= D \oplus D$  and $DG= DG^{(1)} \oplus DG^{(2)}$
 \item[2)] $DG(\cdot,\cdot) =0$
   \item[3)] $\tr_{g}H^2=0$
 \item[4)] $(D_{\cdot}DG)_{\cdot}(\eta,\eta')=D_{\cdot_{1}}G(\eta,\cdot_{2})D_{\cdot_{1}}G(\eta',\cdot_{2}),$ for all  $\eta, \eta' \in \mathfrak{G}$ 
 \item [5)] $D|[,]|^{2}=0$, $D|DG|^{2}=0$ and $D|H^{\mathfrak{G}}|^{2}=0.$
\end{enumerate}
 In the case when $\mathcal{G}$ is nonabelian, we also have:
  \begin{enumerate}
   \item[6)] $DG|_{Z(\mathfrak{G})}=0$
  \item[7)] $(D_{\cdot}DG)_{\cdot}(\eta,\eta')=\frac{1}{2}|DG|^{2}G(\eta,\eta'),$ for all  $\eta, \eta' \in Z(\mathfrak{G})^{\perp}$. 
  \end{enumerate}
   \end{prop}

   \begin{proof} 
     We will be using the following direct consequences of the condition $S_{B}=0$: 
     \begin{align}
      & (D_{\cdot}DG)_{\cdot}(*_{1},*_{2})-D_{\cdot_{1}}G(*_{1},\cdot_{2})D_{\cdot_{1}}G(*_{2},\cdot_{2}) +H(v,*_{1},\cdot)H(v,*_{2},\cdot)=0 \label{EQSB1}
      \\& |[,]|^{2}(D_{v}G(\cdot,\cdot)-2D_{v}G(\eta_{3},\eta_{3}))^{2}=0 \label{EQSB2}
     \\& |[,]|^{2}(D_{v}G(\cdot,\eta_{3}) D_{v}G(\cdot,\eta_{3}) - |DG|_{Z(\mathfrak{G})}^{2}  )=0. \label{EQSB3}
     \end{align}

   To prove Parts 2), 3) and 4), by (\ref{EQSB1}), 
          \begin{align}
     0 &= (D_{v}DG)_{v}(\cdot,\cdot)- |DG|^{2} +H(v, \cdot_{1},\cdot_{2})H(v,\cdot_{1},\cdot_{2})  \nonumber
          \\& =-d^{*}(DG(\cdot,\cdot)) + \tr_{g}H^2 \label{EQSB4}
           \end{align}
        and hence $\int_{M}\tr_{g}H^2 dV_{g}=0$, which implies $\tr_{g}H^2=0.$ Plugging this back into (\ref{EQSB4}), gives
           \[d^{*}(DG(\cdot,\cdot))=0\]
         and using Proposition \ref{propGRAD}, which states that $DG(\cdot,\cdot)=dh$ for some function $h$ on $M$, shows that $d^{*}dh=0$ and hence $dh=DG(\cdot,\cdot)=0$. Moreover, using $\tr_{g}H^2=0$, (\ref{EQSB1}) becomes  \begin{align} \label{EQSB9}(D_{\cdot}DG)_{\cdot}(*_{1},*_{2})=D_{\cdot_{1}}G(*_{1},\cdot_{2})D_{\cdot_{1}}G(*_{2},\cdot_{2}).\end{align}
         
        Next, we have
         \begin{align}
         & D_{v}|DG|^{2}= 2\lp{D_{v}D_{v}G(\cdot_{1},\cdot_{2})- D_{v}G(\cdot_{1},\cdot_{3})D_{v}G(\cdot_{2},\cdot_{3}) }D_{v}G(\cdot_{1},\cdot_{2}) =0 \label{EQSB6}
\\& D|H^{\mathfrak{G}}|^{2}= - |H^{\mathfrak{G}}|^{2}DG(\cdot,\cdot)=0.  \label{EQSB7}
 \end{align}
 (\ref{EQSB6}) follows from (\ref{EQDDG}) and (\ref{EQSB9}); (\ref{EQSB7})  uses $DG(\cdot,\cdot)=0$ and (\ref{EQDH12}), which in turn uses Proposition \ref{PROPdH} and $dH=0$.
         
       We now assume that $\mathcal{G}$ is nonabelian.  
         First, plugging $DG(\cdot,\cdot)=0$ into (\ref{EQSB2}) yields $DG|_{Z(\mathfrak{G})}=0$. Next, by (\ref{EQSB3}), we obtain $DG(\eta_{1},\eta_{3})=0$ and $DG(\eta_{2},\eta_{3})=0$, where $\{\eta_{i}\}_{i=1,2,3}$ is an orthonormal basis for $(\mathfrak{G}|_{m},G|_{m})$ such that $\eta_{3} \in Z(\mathfrak{G})$ and $m \in M$.  Hence $DG(\eta,\eta')=0$, for all $\eta \in Z(\mathfrak{G})^{\perp}$ and $\eta' \in Z(\mathfrak{G}).$ 
     
     We will now use this  to prove $D= D \oplus D$, based on the $G_{t}-$orthogonal splitting $\mathfrak{G}=Z(\mathfrak{G})^{\perp} \oplus Z(\mathfrak{G})$. First note by Proposition \ref{PROPZ}, $D$ preserves sections of $Z(\mathfrak{G})$. Now let $\eta$ and $\eta'$ be sections of $Z(\mathfrak{G})^{\perp}$ and  $Z(\mathfrak{G})$, respectively, and consider 
    \begin{align*}
    0&= D(G(\eta,\eta') ) \\ &= DG(\eta, \eta') +G(D\eta,\eta') +G(\eta,D\eta') \\ & =G(D\eta,\eta'),
    \end{align*}
    where we used that $DG(\eta, \eta')=0$ and $G(\eta,D\eta')=0$. Hence $D\eta$ is a section of $Z(\mathfrak{G})^{\perp}$.
    
 To prove Part 7), using (\ref{EQSB9}) we have
  \begin{align*}
   (D_{\cdot}DG)_{\cdot}(\eta,\eta')=D_{\cdot_{1}}G(\eta,\cdot_{2})D_{\cdot_{1}}G(\eta',\cdot_{2})=  \frac{1}{2}|DG|^{2}_{Z(\mathfrak{G})^{\perp}}G(\eta,\eta'),
  \end{align*}
  for all $\eta, \eta' \in Z(\mathfrak{G})^{\perp}$.  One may check the second equality using an orthonormal basis $\{ \eta_{1},\eta_{2}\}$ for $(Z(\mathfrak{G})^{\perp}|_{m}, G|_{m})$, where $m \in M$, together with the relations $DG(\eta_{1},\eta_{1})+ DG(\eta_{2},\eta_{2})=0$ and $DG(\eta,\eta'')=0,$ for all $\eta \in Z(\mathfrak{G})^{\perp} $ and $\eta'' \in Z(\mathfrak{G})$. (The equality $DG(\eta_{1},\eta_{1})+ DG(\eta_{2},\eta_{2})=0$ follows from $DG(\cdot,\cdot)=0$ and $DG|_{Z(\mathfrak{G})}=0$.)
  
Lastly, we have
  \begin{align*}
  D|[,]|^{2}=- |[,]|^{2}(DG(\cdot,\cdot)-2DG(\eta_{3},\eta_{3})) =0,
  \end{align*}
 which follows from (\ref{EQDBRA}) and  (\ref{EQSB2}).
\end{proof}
 
 \begin{prop} \label{PROPSB2} Let $(\bar{g}(t), \bar{H}(t))$ be a time-dependent invariant metric and closed three form satisfying (\ref{EQ}) on a $\mathcal{G}$--principal bundle $P \rightarrow M$, where $t\in (0,\infty)$, $\mathcal{G}$ is a three dimensional, nilpotent Lie group and $M$ is one-dimensional and compact. If $S_{B}=0$ then in addition to the conclusions of Proposition \ref{PROPSB1}, the following holds true:
   \begin{align*} 
   &1) \ Z(\mathfrak{G})^{\perp} \text{ is independent of } t
 \\&2) \ \frac{\del A}{\del t}=0 \text{ and } \frac{dD}{dt}=0.
   \end{align*}
\end{prop}

\begin{proof}

To prove Part 1), let $\eta_{t}$ be a $t$--dependent element of  $\mathfrak{G}$ and $\eta' \in Z(\mathfrak{G})$ be such that $G_{t}(\eta_{t},\eta')=0$, for all $t \in (0,\infty)$. We will first show $\frac{\partial G}{\partial t}(\eta,\eta')= 0$. Using (\ref{EQ}), we have
\begin{align*}
 &\frac{\partial G}{\partial t}(\eta,\eta')
 \\ &=(D_{\cdot}DG)_{\cdot}(\eta,\eta')- D_{\cdot}G(\eta, \cdot_{1})D_{\cdot}G(\eta',\cdot_{1})
 \\ &  +G([\cdot,\eta],[\cdot,\eta']) -\frac{1}{2}G([\cdot_{1},\cdot_{2}],\eta)G([\cdot_{1},\cdot_{2}],\eta') +\frac{1}{2}H^2(\eta,\eta').
\end{align*}
To analyze these terms, consider the $t$-dependent, $G$-orthogonal splitting $\mathfrak{G}= Z(\mathfrak{G})^{\perp} \oplus Z(\mathfrak{G})$ and set $G= G^{(1)} \oplus G^{(2)}.$ Since $D= D \oplus D$ and $DG=DG^{(1)} \oplus DG^{(2)}$ by Proposition \ref{PROPSB1}, the first two terms in the above expression for $\frac{\partial G}{\partial t}(\eta,\eta')$ are zero.  The third term $G([\cdot,\eta],[\cdot,\eta'])=0$ because $\eta' \in Z(\mathfrak{G})$ and the fourth term $-\frac{1}{2}G([\cdot_{1},\cdot_{2}],\eta)G([\cdot_{1},\cdot_{2}],\eta')=0$ because $[\mathfrak{G},\mathfrak{G}] \subset Z(\mathfrak{G})$ and $\eta \in Z(\mathfrak{G})^{\perp}$.

To show the fifth term, $\frac{1}{2}H^2(\eta,\eta')$, is zero,  let $v \in T_{m}M$ satisfy $g(v,v)=1$,  for some $m \in M$, and let $\{\eta_{i}\}_{i=1,2,3}$ be an orthonormal basis for $(\mathfrak{G}|_{m}, G|_{m})$. Using that $\tr_{g}H^{2}=0$ by Proposition \ref{PROPSB1},  and that $\mathfrak{G}|_{m}$ and $M$ are, respectively, three and one dimensional, we have
\begin{align*}
\frac{1}{2}H^2(\eta,\eta')
&= \frac{1}{2}H(\eta, \eta_{i},\eta_{j})H(\eta', \eta_{i},\eta_{j})
  +H(\eta, v, \eta_{i})H(\eta', v, \eta_{i}) =0.
\end{align*}
Hence $\frac{\partial G}{\partial t}(\eta,\eta')= 0$. 
  
  Now to prove that $Z(\mathfrak{G})^{\perp}$ is independent of $t$, let $\epsilon_{0}>0$ and $\eta_{0} \in \mathfrak{G}$ satisfy $G|_{t=\epsilon_{0}}(\eta_{0},\eta')=0,$ for all $\eta' \in Z(\mathfrak{G})$. Given a nonzero element $\eta' \in Z(\mathfrak{G})$, set $\eta=\eta_{0} -\frac{G(\eta_{0},\eta')}{G(\eta',\eta')}\eta'$, so that $G(\eta,\eta')=0$ for all $t>0$. Since $\frac{\partial G}{\partial t}(\eta,\eta')= 0$, we have
  \begin{align*}
  0=\frac{d}{dt}G(\eta,\eta')= G\lp{\frac{d\eta}{dt},\eta'}=-G(\eta',\eta')\frac{d}{dt}\lp{\frac{G(\eta_{0},\eta')}{G(\eta',\eta')}},
  \end{align*}
 which implies \[\frac{G_{t}(\eta_{0},\eta')}{G_{t}(\eta',\eta')}= \frac{G_{t=\epsilon_{0}}(\eta_{0},\eta')}{G_{t=\epsilon_{0}}(\eta',\eta')}=0,\] for all $t>0$.  Hence $G(\eta_{0},\eta')=0$ for all $\eta' \in Z(\mathfrak{G}) $  and $t>0$. This completes the proof of Part 1). 
 
 To prove Part 2), let $v\in T_{m}M$, $\eta \in \mathfrak{G}|_{m}$, for some $m \in M$, and  let $\{\eta_{i}\}_{i=1,2,3}$ be an orthonormal basis for $(\mathfrak{G}|_{m}, G|_{m})$ such that 
$\eta_{3} \in Z(\mathfrak{G})$. By (\ref{EQ}),
\begin{align*}
 G\lp{\frac{\del A}{\del t}v, \eta} &= D_{v}G([\cdot,\eta],\cdot) + \frac{1}{2}H^2(v,\eta).
		       \end{align*}
The term $D_{v}G([\cdot,\eta],\cdot)= D_{v}G([\eta_{3},\eta],\eta_{3})=0$ because $\eta_{3} \in Z(\mathfrak{G})$, $[\mathfrak{G},\mathfrak{G}] \subset Z(\mathfrak{G})$ and by Proposition \ref{PROPSB1}, $DG(\eta'',\eta')=0$, for all $\eta'' \in Z(\mathfrak{G})^{\perp}$ and $\eta' \in Z(\mathfrak{G})$. Using that $\tr_{g}H^2=0$, which is true by Proposition \ref{PROPSB1},  the term $\frac{1}{2}H^2(v,\eta)=\frac{1}{2}H(v,\eta_{i},\eta_{j})H(\eta,\eta_{i},\eta_{j})=0 $.  This then shows that $\frac{\del A}{\del t}=0.$

Letting $v\in TM$ and $\eta \in  \Gamma(\mathfrak{G})$, we then have \[ \frac{d}{dt}\lp{D_{v}\eta}=-\lb{\frac{\partial A}{\partial t}v, \eta}=0,\]  which proves that $\frac{dD}{dt}=0.$
     \end{proof}

\subsection{Convergence Results}\label{SECCONV}

 We will now use the $\mathcal{I}$--functional of Section \ref{SECFUNC} to study blowdown limits of invariant generalized Ricci flows.
 
\begin{thm}  \label{THMLIMIT}
 Let $(\bar{g}(t), \bar{H}(t))$ be a time-dependent invariant metric and closed three-form satisfying (\ref{EQ}) on a $\mathcal{G}$--principal bundle $P \rightarrow M$, where $t\in (0,\infty)$, $\mathcal{G}$ is a three dimensional,  nilpotent Lie group and $M$ is one-dimensional and compact. Define $\bar{g}_{k}(t)=s_{k}^{-1}\bar{g}(s_{k}t)$ and $\bar{H}_{k}(t)=s_{k}^{-1}\bar{H}(s_{k}t)$, where $\{s_{k}\}$ is a sequence of positive numbers such that $\lim s_{k}=\infty$. We suppose there exists a $\mathcal{G}$--principal bundle $P_{\infty} \rightarrow M_{\infty}$, over a compact and one-dimensional base manifold, together with bundle isomorphisms $\psi_{k}:P_{\infty} \rightarrow P$  such that $(\psi_{k}^{*}(\bar{g}_{k})(t),\psi_{k}^{*}(\bar{H}_{k})(t))$ uniformly converges to $(\bar{g}_{\infty}(t),\bar{H}_{\infty}(t))$ on $P_{\infty} \times [j^{-1},j]$, for all $j>0$. The following holds true in the limit: 
\begin{enumerate} 
 \item[1)] $\mathfrak{G}_{0}:=Z(\mathfrak{G})^{\perp}$ is independent of  $t$
 \item[2)] under the splitting  $\mathfrak{G}=\mathfrak{G}_{0} \oplus Z(\mathfrak{G})$,  $ D= D \oplus D$ and is independent of  $t$  
  \item[3)]  $\frac{\partial A}{\partial t}=0$ 
 \item[4)] $H= H^{\mathfrak{G}}$ and $\frac{\del H}{\del t}=0$
  \item[5)] $DG(\cdot,\cdot) =0$
\item[6)] $(D_{\cdot}DG)_{\cdot}(\eta,\eta')=D_{\cdot_{1}}G(\eta,\cdot_{2})D_{\cdot_{1}}G(\eta',\cdot_{2})$, for all $\eta, \eta' \in \mathfrak{G}$ 
 \item[7)] $D|[,]|^{2}=0$ and  $D|H^{\mathfrak{G}}|^{2}=0$
\item[8)] $|DG|^{2}=\frac{2}{t}$ and   $g= tg|_{t=1}$.
  \end{enumerate}
  In the case when $\mathcal{G}$ is nonabelian, we also have:
  \begin{enumerate}
  \item[9)] $DG|_{Z(\mathfrak{G})}=0$ 
     \item[10)] $(D_{\cdot}DG)_{\cdot}(\eta,\eta')=\frac{G(\eta,\eta')}{t}, \text{ for all } \eta, \eta' \in \mathfrak{G}_{0}$. 
  \end{enumerate}
  
\end{thm}

\begin{proof}
Using the properties of the functional defined in Definition \ref{d:energy}, we have 
\begin{align*}
\mathcal{I}(\bar{g}_{\infty}(t),\bar{H}_{\infty}(t),t)
  				&=\lim_{k\rightarrow \infty}\mathcal{I}(\psi_{k}^{*}(\bar{g}_{k})(t),\psi_{k}^{*}(\bar{H}_{k})(t),t)
							\\&=	\lim_{k\rightarrow \infty}\mathcal{I}(\bar{g}_{k}(t),\bar{H}_{k}(t),t)
		\\&= \lim_{k\rightarrow \infty}\mathcal{I}(s_{k}^{-1}\bar{g}(s_{k}t),s_{k}^{-1}\bar{H}(s_{k}t),t)
		   \\&= \lim_{k\rightarrow \infty}\mathcal{I}(\bar{g}(s_{k}t),\bar{H}(s_{k}t),s_{k}t)
		       \\& = \lim_{ t\rightarrow \infty}\mathcal{I}(\bar{g}(t),\bar{H}(t),t).
						\end{align*}
	Note the existence of the last limit follows from the existence of the second to last limit and the fact that $\mathcal{I}(\bar{g}(t),\bar{H}(t),t)$ is a nonincreasing function of $t$ by Proposition \ref{p:Imonotone}. (One may also directly use that $\mathcal{I} \geq 0$.)
	
	Hence,  $\mathcal{I}(\bar{g}_{\infty}(t),\bar{H}_{\infty}(t),t)$ is constant in $t$. Using Proposition \ref{p:Imonotone}, 
	  the following then holds true in the limit: 
	     \begin{align*}
					   0& =  \frac{d}{dt}\mathcal{I}(\bar{g}(t), \bar{H}(t),t)
					      \\&= -t\int_{M}S_{B}\frac{dV_{g}}{\sqrt{t}} -\frac{t}{4}\int_{M}\lp{|DG|^{2} + \tr_{g}H^2-\frac{2}{t}}^{2}\frac{dV_{g}}{\sqrt{t}}.     					      
\end{align*}
Since $S_{B}>0$, $|DG|^{2} + \tr_{g}H^2=\frac{2}{t}$ and $S_{B}=0$ in the limit. Combining the results of Propositions \ref{PROPSB1} and \ref{PROPSB2}, proves the theorem.
\end{proof}

\begin{cor}\label{CORLIMIT} Let $(\bar{g}_{\infty}(t),\bar{H}_{\infty}(t))$ be the limiting generalized Ricci flow solution given in Theorem \ref{THMLIMIT}. 

1) The solution  in the limit satisfies the following differential equations:   
\begin{gather}
\begin{split}
    &   \frac{d}{dt}|[,]|^{2}= -\frac{3}{2}|[,]|^{4} -\frac{1}{6}|H^{\mathfrak{G}}|^{2}|[,]|^{2}
 \\&   \frac{d}{dt}|H^{\mathfrak{G}}|^{2}= -\frac{1}{2}|H^{\mathfrak{G}}|^{4} -\frac{1}{2}|H^{\mathfrak{G}}|^{2}|[,]|^{2}    
 \\&  \frac{\partial G}{\partial t}(\eta,\eta')=\lp{\frac{1}{2}|[,]|^{2}+\frac{1}{6} |H^{\mathfrak{G}}|^{2}}G(\eta,\eta'), \text{ for all } \eta, \eta' \in \mathfrak{G}_{0}
\\& \frac{\partial G}{\partial t}(\eta,\eta')=\lp{-\frac{1}{2}|[,]|^{2}+\frac{1}{6} |H^{\mathfrak{G}}|^{2}}G(\eta,\eta'), \text{ for all } \eta, \eta' \in Z(\mathfrak{G}).
\end{split}
\end{gather}

2)  If $H^{\mathfrak{G}}=0$ and $\mathcal{G}$ is abelian then $\frac{\del G}{\del t}=0$ in the limit.\\

3) If $H^{\mathfrak{G}}=0$ and $\mathcal{G}$ is nonabelian then the following holds true in the limit:
\begin{align*}
      &  \quad  a) \ |[,]|^{2}= \frac{1}{at+C},\text{ where } a=\frac{3}{2} \text{ and } C \geq 0 
        \\& \quad b) \ G(\eta,\eta')=\lp{\frac{at+C}{a+C}}^{\frac{1}{3}}G|_{t=1}        
                           (\eta,\eta'), \text{ for all } \eta,\eta' \in \mathfrak{G}_{0}
        \\& \quad c) \ G(\eta,\eta')=\lp{\frac{at+C}{a+C}}^{-\frac{1}{3}}G|_{t=1}
                           (\eta,\eta'), \text{ for all } \eta,\eta' \in Z(\mathfrak{G}).
\end{align*}
\end{cor}
   
\begin{proof}
The proof immediately follows from Propositions \ref{PROPbra2} and \ref{PROPEVH}, Theorem \ref{THMLIMIT} and (\ref{EQ}).
\end{proof}
  
\bibliographystyle{plain}

 \end{document}